\documentclass[12pt,reqno]{amsart}

\usepackage{amssymb,amsmath,graphicx,amsfonts,euscript}
\usepackage{color}

\newtheorem{thm}{Theorem}[section]

\newtheorem{cor}[thm]{Corollary}
\newtheorem{prop}[thm]{Proposition}
\newtheorem{define}[thm]{Definition}

\newtheorem{lemma}[thm]{Lemma}

\newcommand{\p}{\partial}

\numberwithin{equation}{section}

\begin{document}
\title[Supercritical 2D Boussinesq equations]
{The 2D Boussinesq equations with logarithmically supercritical velocities}
\author[Dongho Chae and Jiahong Wu]{Dongho Chae$^{1}$ and Jiahong Wu$^{2}$}
\address{$^1$ Department of Mathematics,
Sungkyunkwan University,
Suwon 440-746, Korea}

\email{chae@skku.edu}

\address{$^3$Department of Mathematics,
Oklahoma State University,
401 Mathematical Sciences,
Stillwater, OK 74078, USA.}

\email{jiahong@math.okstate.edu}

\begin{abstract}
This paper investigates the global (in time) regularity of solutions to a system of equations that generalize the vorticity formulation of the 2D Boussinesq-Navier-Stokes equations. The velocity $u$ in this system is related to the vorticity $\omega$ through the relations $u=\nabla^\perp \psi$ and $\Delta \psi = \Lambda^\sigma (\log(I-\Delta))^\gamma \omega$, which reduces to the standard velocity-vorticity relation when $\sigma=\gamma=0$. When either $\sigma>0$ or $\gamma>0$, the velocity $u$ is more singular. The ``quasi-velocity" $v$ determined by $\nabla\times v =\omega$ satisfies an equation of very special structure. This paper establishes the global regularity and uniqueness of solutions for the case when $\sigma=0$ and $\gamma\ge 0$. In addition, the vorticity $\omega$ is shown to be globally bounded in several functional settings such as $L^2$ for $\sigma>0$ in a suitable range.
\end{abstract}

\maketitle

\section{Introduction}

This paper aims at the global regularity problem on the generalized 2D Boussinesq equations
\begin{equation}\label{GBou}
\left\{
\begin{array}{l}
\displaystyle  \p_t \omega + u\cdot\nabla\omega + \Lambda \omega = \theta_{x_1},
\\
\displaystyle u =\nabla^\perp \psi \equiv (-\p_{x_2}, \p_{x_1}) \psi, \quad  \Delta \psi = \Lambda^\sigma (\log(I-\Delta))^\gamma \omega,
\\
\displaystyle  \p_t \theta + u\cdot \nabla \theta  =0,
\end{array}
\right.
\end{equation}
where $\omega=\omega(x,t)$, $\psi=\psi(x,t)$ and $\theta=\theta(x,t)$ are scalar functions of $x\in \mathbb{R}^2$ and $t\ge 0$, $u=u(x,t): \mathbb{R}^2\to \mathbb{R}^2$ is a vector field, $\sigma\ge 0$ and $\gamma\ge 0$ are real parameters, and $\Lambda=(-\Delta)^{\frac12}$ and $\Lambda^\sigma$ are Fourier multiplier operators with
$$
\widehat{\Lambda^\sigma f} (\xi) = |\xi|^\sigma \widehat{f}(\xi).
$$
For a given initial data
\begin{equation} \label{IV}
\omega(x,0) =\omega_0(x), \quad \theta(x,0) =\theta_0(x),
\end{equation}
we would like to determine whether or not the corresponding solution is global in time.
\vskip .1in
The model studied here can be regarded as a generalization of the vorticity formulation of the 2D Boussinesq equations
\begin{equation}\label{BouV}
\left\{
\begin{array}{l}
\displaystyle  \p_t u + u\cdot\nabla u =\nu \Delta u  -\nabla p
+ \theta \mathbf{e}_2,
\\
\displaystyle \nabla \cdot u =0
\\
\displaystyle  \p_t \theta + u\cdot \nabla \theta =\kappa \Delta \theta,
\end{array}
\right.
\end{equation}
where $\nu\ge 0$ and $\kappa\ge 0$ are real parameters and $\mathbf{e}_2=(0,1)$ is the unit vector in the $x_2$-direction. Boussinseq type equations model geophysical flows such as atmospheric fronts and ocean circulations (see, e.g., \cite{Maj,Pe}). Mathematically the 2D Boussinesq equations serve as a lower-dimensional model of the 3D hydrodynamics equations. In fact, the 2D Boussinesq equations retain some key features of the 3D Euler and Navier-Stokes equations such as the vortex stretching mechanism and, as pointed out in \cite{MB}, the inviscid 2D Boussinesq equations are identical to the Euler equations for the 3D axisymmetric swirling flows outside the symmetry axis. It is hoped that the study of the 2D Boussinesq equations may shed light on the global regularity problem concerning the 3D Euler and Navier-Stokes equations.

\vskip .1in
The global regularity problem for the 2D Boussinesq equations have been extensively studied and important progress has been made (see, e.g., \cite{AbHm,ACW10,ACW11,CaDi, CaoWu1,Ch,CV,DP1,DP2,DP3,HmKe1,HmKe2,HKR1,HKR2,HL,LLT,MX}).
When $\nu>0$, $\kappa>0$, (\ref{BouV}) with any sufficiently smooth data has a global solution (see, e.g., \cite{CaDi}). In the case of inviscid Boussinesq equations, namely (\ref{BouV}) with $\nu=\kappa=0$, the global regularity problem remains outstandingly open. The global regularity for the case $\nu>0$ and $\kappa=0$ was obtained by Chae \cite{Ch} and by Hou and Li \cite{HL}. The case when $\nu=0$ and $\kappa>0$ was dealt with in \cite{Ch}.  Their results successfully resolved one of the open problems proposed by Moffatt \cite{Mof}. Further progress on these two cases was recently made  by Hmidi, Keraani and Rousset, who were able to establish the global regularity even when the full Laplacian dissipation is replaced by the critical dissipation represented in terms of $\sqrt{-\Delta}$ (\cite{HKR1},\cite{HKR2}). The work of Hmidi, Keraani and Rousset was further generalized by Miao and Xue to accommodate both fractional dissipation and fractional thermal diffusion \cite{MX}. In a very recent preprint \cite{CV} Constantin and Vicol applied the nonlinear maximum principle for linear nonlocal operators to obtain another global regularity result when the fractional powers of the Laplacians for the dissipation and thermal diffusion obey certain conditions.  The global well-posedness for the anisotropic Boussinesq equations with horizontal dissipation or thermal diffusion was first studied by Danchin and Paicu \cite{DP3}. Recently Larios, Lunasin and Titi \cite{LLT} further investigated the Boussinesq equations with horizontal dissipation via more elementary approaches and re-established some results of Danchin and Paicu under milder assumptions. The global regularity problem for the 2D Boussinesq equations with vertical dissipation has been studied by Adhikari, Cao and Wu \cite{ACW10,ACW11} and was successively resolved by Cao and Wu \cite{CaoWu1}.

\vskip .1in
We first point out that the vorticity equation in (\ref{GBou}) does have a corresponding velocity formulation
\begin{equation}\label{vv0}
\p_t v + u\cdot\nabla v -\sum_{j=1}^2 u_j \nabla v_j +  \Lambda v=-\nabla p + \theta \mathbf{e}_2,
\end{equation}
where $v$ satisfies
$$
\nabla\cdot v =0, \quad u= \Lambda^\sigma (\log(I-\Delta))^\gamma v \quad\mbox{or} \quad \nabla\times v =\omega.
$$
When $\sigma=\gamma=0$, $u=v$ and (\ref{vv0}) reduces to the Boussinesq velocity equation after redefining the pressure by $p-\frac12 |u|^2$. The details of the derivation is left in the second section.

\vskip .1in
Our motivation for studying the global regularity of (\ref{GBou}) comes from two different sources: the first being the models generalizing the surface quasi-geostrophic equation and the 2D hydrodynamics equations (see, e.g., \cite{ChCW,ChCW2,CCD,CIW,DKV,Kinew1,MiXu2} and the second being the the Boussinesq-Navier-Stokes system with critical dissipation \cite{HKR1}. In a recent work \cite{HKR1} Hmidi, Keraani and Rousset successfully established the global regularity of the Boussinesq-Navier-Stokes system with critical dissipation, namely (\ref{GBou}) with $\sigma=0$ and $\gamma=0$. Their key idea is to consider the combined quantity
$$
G= \omega- \mathcal{R} \theta,
$$
which satisfies
\begin{equation}\label{ggg}
\partial_t G + u\cdot\nabla G + \Lambda G = - [\mathcal{R}, u\cdot\nabla] \theta.
\end{equation}
Here $\mathcal{R}=\Lambda^{-1}\partial_{x_1}$ stands for a Riesz transform and the brackets denote the commutator. The advantage of (\ref{ggg}) is that we can avoid evaluate the derivatives of $\theta$ when estimating the Lebesgue norm of $G$. This approach is also useful in the handling of the generalized Boussinesq equations (\ref{GBou}).

\vskip .1in
Our goal here is to extend their work to cover more singular velocities and explore how far one can go beyond the critical case.  When either $\sigma>0$ or $\gamma>0$, the corresponding velocity field $u$ is more singular. We are able to obtain the global regularity and uniqueness of solutions to (\ref{GBou}) for the special case when $\sigma=0$ and $\gamma>0$.

\begin{thm} \label{major1}
Consider the generalized Boussinesq equations (\ref{GBou}) with $\sigma=0$ and $\gamma\ge 0$. Assume the initial data $(\omega_0, \theta_0)$ satisfies $$
\omega_0\in L^2 \cap L^q \cap B^{0,\gamma}_{\infty,1}, \quad \theta_0\in L^2\cap B^{0,\gamma}_{\infty,1}
$$
for some $q > 2$. Then (\ref{GBou}) has a unique global solution $(\omega, \theta)$ satisfying, for any $t>0$,
$$
\omega\in L^2 \cap L^q \cap L^1_t B^{0,\gamma}_{\infty,1}, \quad \theta \in L^2\cap L^\infty \cap L^1_t B^{0,\gamma}_{\infty,1}.
$$
\end{thm}

Here $B^{0,\gamma}_{\infty,1}$ is a space of Besov type and its definition is provided in the Appendix. Although it is not clear if this global regularity result still holds for the more singular case when $\sigma>0$, we can still show that the $L^2$-norm of the vorticity $\omega$ is bounded at any time for $0\le \sigma <\frac12$ and $\gamma\ge 0$. More precisely, we have the following theorem.
\begin{thm} \label{globalw}
Consider (\ref{GBou}) with $0\le \sigma <\frac12$ and $\gamma \ge 0$. Assume $(\omega_0, \theta_0)$ satisfies the conditions stated in Theorem \ref{major1}.
Let $(\omega, \theta)$ be the corresponding solution. Then, for any $t>0$,
$$
\|\omega(t)\|_{L^2} \le B(t)
$$
for a smooth function $B(t)$ of $t$ depending on the initial data only. In addition, $G$ satisfies the basic energy bound
\begin{equation}\label{Gbd}
\|G(t)\|_{L^2}^2 + \int_0^t \|\Lambda^{\frac12} G\|_{L^2}^2 dt \le B(t).
\end{equation}
\end{thm}

Further regularity can also be established for certain $\sigma>0$. In fact, $\|\omega\|_{L^q}$ for $q\in (2, \frac4{1+2\sigma}]$ is also globally bounded
in time when $0\le \sigma <\frac12$ and $\gamma \ge 0$ ($q\not =\frac4{1+2\sigma}$
when $\gamma>0$). In addition, for $0\le \sigma<\frac14$ and $\gamma\ge 0$,
the space-time norm
$\widetilde{L}^r_t B^{s}_{q,1}$ of $G$ is also bounded for any $t>0$. The precise statement is given in Theorem \ref{ltb}. This bound especially implies that $G$ is in $L^1_tL^\infty_x$. However, we need to assume $\sigma=0$ in order to obtain the global bounds for $\omega$ and $\theta$ in  $L^1_tL^\infty_x$.

\vskip .1in
The rest of this paper is divided into four sections. The second section derives the velocity formulation of a generalized Boussinesq vorticity equation. The third section proves the global $L^2$ vorticity bound stated in Theorem \ref{globalw}. It requires a commutator estimate involving the Riesz transform $\mathcal{R}$. Section \ref{exist} proves the aforementioned global regularity bounds and part of Theorem \ref{major1} while Section \ref{uniq} establishes the uniqueness part of Theorem \ref{major1}. Throughout the rest of this paper, $B(t)$'s denote bounds that depend on $t$ and the initial data.

\vskip .4in
\section{Derivation of the velocity equation}

This section derives the velocity formulation for the
generalized 2D Boussineq vorticity equation given by
\begin{equation} \label{gbb}
\left\{
\begin{array}{l}
\displaystyle  \p_t \omega + u\cdot\nabla\omega + \nu \Lambda^\alpha \omega = \theta_{x_1},
\\
\displaystyle u =\nabla^\perp \psi \equiv (-\p_{x_2}, \p_{x_1}) \psi, \quad  \Delta \psi = P(\Lambda)\omega,
\\
\displaystyle  \p_t \theta + u\cdot \nabla \theta + \kappa \Lambda^\beta \theta =0
\end{array}
\right.
\end{equation}
where $\nu\ge 0$, $\kappa\ge 0$, $0<\alpha\le 1$, $0<\beta\le 1$ are real parameters, and $P(\Lambda)$ is a Fourier multiplier operator  with
$$
 \widehat{P(\Lambda) f} (\xi) = P(|\xi|) \widehat{f}(\xi).
$$
Clearly, (\ref{GBou}) is a special case of (\ref{gbb}). A special consequence of Theorem \ref{ved} below is the derivation of (\ref{vv0}).

\begin{thm} \label{ved}
For classical solutions of \eqref{gbb} that decay sufficiently fast as $|x|\to \infty$, \eqref{gbb} is equivalent to the following equations
\begin{equation}\label{vueq}
\left\{
\begin{array}{l}
\displaystyle \p_t v + u^\perp (\nabla^\perp \cdot v) + \nu \Lambda^\alpha v=-\nabla p + \theta \mathbf{e}_2,
\\
\displaystyle \nabla \cdot v=0,\quad u = P(\Lambda) v,
\\
\displaystyle  \p_t \theta + u\cdot \nabla \theta + \kappa \Lambda^\beta \theta =0.
\end{array}
\right.
\end{equation}
In addition, the equation for $v$ can be written in the more familiar form
\begin{equation}\label{vv}
\p_t v + u\cdot\nabla v -\sum_{j=1}^2 u_j \nabla v_j + \nu \Lambda^\alpha v=-\nabla p + \theta \mathbf{e}_2.
\end{equation}
\end{thm}

\begin{proof}
It follows from the second equation in (\ref{gbb}) that
$$
u =\nabla^\perp \Delta^{-1} P(\Lambda) \omega, \quad \nabla\times u = \nabla^\perp \cdot u = \Delta \psi =P(\Lambda)\omega.
$$
Therefore, if we set
\begin{equation}\label{vu}
v = P(\Lambda)^{-1} u,
\end{equation}
then
\begin{equation}\label{v1}
v=\nabla^\perp \Delta^{-1}\omega \quad\mbox{and}\quad  \omega = P(\Lambda)^{-1} \nabla\times u  = \nabla\times v.
\end{equation}
Applying $\nabla^\perp \Delta^{-1}$ to the first equation in (\ref{gbb}), we obtain
$$
\p_t v + \Delta^{-1} \nabla^\perp (u\cdot\nabla \omega) + \nu \Lambda^\alpha v
= \Delta^{-1} \nabla^\perp \theta_{x_1}.
$$
To rewrite the nonlinear term, we consider the components of $\nabla^\perp (u\cdot\nabla \omega)$:
\begin{eqnarray*}
-\p_{x_2}  (u\cdot\nabla \omega) &=& -\p_{x_2} (\nabla \cdot(u \omega))\\
&=&-\p_{x_2}(\p_{x_1}(u_1 \omega) + \p_{x_2}(u_2 \omega))\\
&=&\p_{x_1}(-\p_{x_2}(u_1 \omega)) -\Delta (u_2 \omega) + \p_{x_1} (\p_{x_1}(u_2 \omega))\\
&=&-\Delta (u_2 \omega) + \p_{x_1}(-\p_{x_2}(u_1 \omega) +\p_{x_1}(u_2 \omega)),
\end{eqnarray*}
\begin{eqnarray*}
\p_{x_1}  (u\cdot\nabla \omega) &=& \p_{x_1} (\nabla \cdot(u \omega))\\
&=&\p_{x_1}(\p_{x_1}(u_1 \omega) + \p_{x_2}(u_2 \omega))\\
&=&\p_{x_1}\p_{x_1}(u_1 \omega) + \p_{x_2}(\p_{x_1}(u_2 \omega))\\
&=&\Delta(u_1\omega) +  \p_{x_2}(\p_{x_1}(u_2 \omega)-\p_{x_2}(u_1 \omega)).
\end{eqnarray*}
That is,
\begin{equation}\label{t1}
\Delta^{-1} \nabla^\perp (u\cdot\nabla \omega) = u^\perp \omega - \Delta^{-1} \nabla (\nabla\cdot(u^\perp \omega)).
\end{equation}
In addition,
$$
\nabla^\perp \theta_{x_1}
= \left( \begin{array}{c} -\partial_{x_1}\p_{x_2} \theta \\ \p_{x_1}^2 \theta \end{array}
\right) =  \left( \begin{array}{c} \partial_{x_1}(-\p_{x_2} \theta) \\\partial_{x_2}(-\p_{x_2} \theta)  \end{array}
\right) + \left( \begin{array}{c} 0 \\ \Delta \theta \end{array}
\right)
$$
and
\begin{equation}\label{t2}
\Delta^{-1} \nabla^\perp \theta_{x_1} = \theta \mathbf{e}_2 + \Delta^{-1}  \nabla (-\p_{x_2}\theta).
\end{equation}
Inserting (\ref{t1}) and (\ref{t2}) in (\ref{v1}), we obtain, after noting $\omega=\nabla^\perp\cdot v$
\begin{equation} \label{b_veq}
\p_t v + u^\perp (\nabla^\perp \cdot v) + \nu \Lambda^\alpha v=-\nabla p + \theta \mathbf{e}_2
\end{equation}
where
\begin{equation}\label{p1}
p =-\Delta^{-1}\left(\nabla \cdot(u^\perp\nabla^\perp \cdot v)-\p_{x_2}\theta\right).
\end{equation}
Clearly, \eqref{p1} is a simple consequence of \eqref{b_veq} with $\nabla\cdot v=0$. We can rewrite the nonlinear term into more familiar form.  Inserting the identity
$$
 u^\perp (\nabla^\perp \cdot v) = u\cdot\nabla v - \sum_{j=1}^2 u_j \nabla v_j
$$
in \eqref{b_veq}, we find
\begin{equation} \label{b_veqf}
\p_t v + u\cdot\nabla v -\sum_{j=1}^2 u_j \nabla v_j + \nu \Lambda^\alpha v=-\nabla p + \theta \mathbf{e}_2.
\end{equation}
(\ref{vueq}) is a combination of (\ref{b_veq}), (\ref{p1}) and the last equation in (\ref{gbb}). (\ref{vv}) is just (\ref{b_veqf}). This completes the proof of Theorem \ref{ved}.
\end{proof}

\vskip .4in
\section{Global (in time) bound for $\|\omega\|_{L^2}$}
\label{sec:l2w}

This section proves Theorem \ref{globalw}, the global {\it a priori} ${L^2}$-bound for the vorticity $\omega$. To do so, one considers the equation for $G=\omega -\mathcal{R }\theta$,
\begin{equation}\label{Geq}
\partial_t G + u\cdot\nabla G + \Lambda G = - [\mathcal{R}, u\cdot\nabla] \theta.
\end{equation}
Clearly, in order to control $\|G\|_{L^2}$, we need a bound for the commutator $[\mathcal{R}, u\cdot\nabla] \theta$.  For this purpose, we start with the following lemma.
\begin{lemma} \label{newl}
Let $p\in [1,\infty]$ and $\delta\in (0,1)$. If $|x|^\delta \phi\in L^1$, $f\in \mathring{B}^\delta_{p,\infty}$ and $g\in L^\infty$, then
\begin{equation}\label{pfg}
\|\phi\ast (fg) -f (\phi\ast g)\|_{L^p} \le C \||x|^\delta \phi\|_{L^1} \|f\|_{\mathring{B}^\delta_{p,\infty}}  \|g\|_{L^\infty}.
\end{equation}
In the case when $\delta=1$, (\ref{pfg}) is replaced by
\begin{equation}\label{pfg1}
\|\phi\ast (fg) -f (\phi\ast g)\|_{L^p} \le C \||x| \phi\|_{L^1} \|\nabla f\|_{L^p} \|g\|_{L^\infty}.
\end{equation}
\end{lemma}

$\mathring{B}^\delta_{p,\infty}$ here denotes a homogeneous Besov space, which is defined in the Appendix. (\ref{pfg1}) was previously obtained in Lemma 3.2 of \cite[p.2153]{HKR1}. Our extension to cover the case for $\delta\in (0,1)$ is necessary in order to deal with the generalized Boussinesq equations (\ref{GBou}). Since now the velocity field $u$ is more singular, namely
$$
u= \nabla^\perp \Delta^{-1} \Lambda^\sigma (\log(I-\Delta))^\gamma \omega,
$$
it is necessary to consider the fractional derivative $\Lambda^{1-\sigma} u$, which, roughly speaking, is more or less $\omega$ when evaluated in a Lebesgue space. When $\sigma>0$, we can no linger control $\nabla u$ in terms of $\omega$, as did in \cite{HKR1}.

\begin{proof} By Minkowski's inequality, for any $p\in [1,\infty]$,
\begin{eqnarray*}
\|\phi\ast (fg) -f (\phi\ast g)\|_{L^p}
&=&\left[\int \left|\int \phi(z) \, (f(x) -f(x-z)) g(x-z) \,dz\right|^p \,dx\right]^{1/p} \\
&\le& \int \left[\int |\phi(z) \, (f(x) -f(x-z)) g(x-z)|^p dx \right]^{1/p} dz\\
&\le& \|g\|_{L^\infty} \int |\phi(z)| \, \|f(\cdot)-f(\cdot-z))\|_{L^p} \,dz\\
&\le& \|g\|_{L^\infty} \sup_{|z|>0} \frac{\|f(\cdot)-f(\cdot-z))\|_{L^p}}{|z|^\delta} \||z|^\delta |\phi(z)|\|_{L^1}
\end{eqnarray*}
\eqref{pfg} then follows from the definition of $\mathring{B}^\delta_{p,\infty}$.
\end{proof}

\vskip .1in
We now present a general proposition that provides an estimate for the commutator
as in \eqref{Geq}. The proof of this proposition is obtained by modifying that of
Proposition 3.3 in \cite{HKR1}. Since the proof is slightly long, we leave it to
the end of this section.

\begin{prop}\label{pp1}
Let $u: \mathbb{R}^d\to \mathbb{R}^d$ be a vector field. Let $\mathcal{R}=\partial_{x_1} \Lambda^{-1}$ denote a Riesz transform. Let $s\in (0,1)$, $s<\delta< 1$, $p\in (1,\infty)$ and $q\in [1,\infty]$. Then
\begin{equation} \label{jj1}
\|[\mathcal{R}, u] F\|_{B^s_{p,q}} \le  C_1\, \|u\|_{\mathring{B}^{\delta}_{p,\infty}}\, \|F\|_{B^{s-\delta}_{\infty,q}}
+ C_2 \sum_{j=-1}^3 \|\Delta_j u\,\Delta_j F\|_{L^p},
\end{equation}
where $C_1$ is a constant depending on $d$, $s$, $\delta$, $p$ and $q$ only and $C_2$ is an absolute constant. When $\delta=1$, $\|u\|_{\mathring{B}^{\delta}_{p,\infty}}$ is replaced by $\|\nabla u\|_{L^p}$.
\end{prop}

\vskip .1in
We now apply Proposition \ref{pp1} to the special case when $u$ is determined by
$\omega$ through the relations in (\ref{GBou}). We obtain a bound for the commutator involved in the equation for $G$, namely \eqref{Geq}.
\begin{cor}\label{bone}
Let $u: \mathbb{R}^2\to \mathbb{R}^2$ be a vector field determined by a scalar function $\omega$ through the relations
\begin{equation}\label{uw}
u =\nabla^\perp \psi, \quad  \Delta \psi = \Lambda^\sigma \left(\log(I-\Delta)\right)^\gamma\omega,
\end{equation}
where $0\le \sigma <\frac12$ and $\gamma\ge 0$ are real parameters. Then,  for any $0\le s < 1-\sigma$, $p\in (1,\infty)$ and $q\in [1,\infty]$,
\begin{equation} \label{spe0}
\|[\mathcal{R}, u] \theta\|_{B^s_{p,q}} \le  C\, \|\omega\|_{L^p} \|\theta\|_{B^{s+\sigma-1}_{\infty,q}} + C\, \|\omega\|_{L^{p_1}} \,\|\theta\|_{L^{p_2}},
\end{equation}
where $p_1$ and $p_2$ satisfy
$$
p_1\in [1,\infty), \quad p_2\in [1,\infty], \quad \frac1{p_1} + \frac1{p_2} = \frac1p + \frac{1-\sigma}{2}
$$
and $C$'s are constants depending on $\sigma$, $\gamma$, $s$, $p$, $q$, $p_1$ and $p_2$. Furthermore, for any $p_3 \ge \frac2{1-s-\sigma}$,
\begin{equation} \label{spe}
\|[\mathcal{R}, u] \theta\|_{H^s} \le  C\,\|\omega\|_{L^2} (\|\theta\|_{L^{p_3}} +\|\theta\|_{L^{\frac{2}{1-\sigma}}}),
\end{equation}
where $C$ is a constant depending on $\sigma$, $s$ and $p_3$ only.
\end{cor}

\begin{proof}[Proof of Corollary \ref{bone}] By Proposition \ref{pp1},
$$
\|[\mathcal{R}, u] \theta\|_{B^s_{p,q}} \le  C\,\|u\|_{\mathring{B}^{\delta}_{p,\infty}}\, \|\theta\|_{B^{s-\delta}_{\infty,q}}
+ C \sum_{j=-1}^3 \|\Delta_j u\,\Delta_j \theta\|_{L^p}.
$$
According to (\ref{uw}),
$$
u = \nabla^\perp \Lambda^{-2+\sigma} \left(\log(I-\Delta)\right)^\gamma\omega.
$$
Since $s+\sigma<1$, we choose $\epsilon>0$ such that $s+\sigma + \epsilon=1$. Then,
\begin{eqnarray*}
\|u\|_{\mathring{B}^{s}_{p,\infty}} &\le& \|\left(\log(I-\Delta)\right)^\gamma\omega\|_{\mathring{B}^{s+\sigma-1}_{p,\infty}} \\
&\le& C\, \|\omega\|_{\mathring{B}^{s+\sigma+\epsilon-1}_{p,\infty}} \\
&\le& C\,\|\omega\|_{L^p}.
\end{eqnarray*}
In addition, for any $-1\le j\le 3$, we have
$$
\|\Delta_j u\,\Delta_j \theta\|_{L^p} \le \|\Delta_j u\|_{L^{q_1}} \, \|\Delta_j \theta\|_{L^{p_2}} \le C\, \|\Lambda^{\sigma-1} \omega\|_{L^{q_1}} \, \| \theta\|_{L^{p_2}}
$$
where $q_1\in(1,\infty), p_2\in [1,\infty]$ and $\frac1{q_1} + \frac1{p_2} =\frac1p$. By Hardy-Littlewood-Sobolev inequality,
$$
\|\Lambda^{\sigma-1} \omega\|_{L^{q_1}} \le C \|\omega\|_{L^{p_1}}.
$$
where $1\le p_1<q_1<\infty$ and $\frac1{q_1} =\frac1{p_1} -\frac{1-\sigma}{2}$. Therefore,
$$
\sum_{j=-1}^3 \|\Delta_j u\,\Delta_j \theta\|_{L^p} \le C\, \|\omega\|_{L^{p_1}} \|\theta\|_{L^{p_2}}
$$
with $p_1$ and $p_2$ satisfying $\frac1{p_1} + \frac1{p_2} = \frac1p + \frac{1-\sigma}{2}$. (\ref{spe}) is obtained by taking $p=q=p_1=2$, $p_2= \frac{2}{1-\sigma}$ in (\ref{spe0}) and applying the embedding relation
$$
L^{p_3} \hookrightarrow B^{s+\sigma-1}_{\infty,2}.
$$
This completes the proof of Corollary \ref{bone}.
\end{proof}

\vskip .1in
With Corollary \ref{bone} at our disposal, we now prove Theorem
\ref{globalw}.

\begin{proof}[Proof of Theorem \ref{globalw}]  Multiplying (\ref{Geq}) by $G$ and integrating over $\mathbb{R}^2$, we obtain
$$
\frac12 \frac{d}{dt}\|G\|_{L^2}^2 + \|\Lambda^{\frac12} G\|_{L^2}^2 =- \int G\, \nabla\cdot [\mathcal{R}, u]\theta\,dx.
$$
By H\"{o}lder's inequality,
\begin{equation}\label{goo1}
\left|\int G\, \nabla\cdot [\mathcal{R}, u]\theta\,dx \right| \le \|\Lambda^{\frac12} G\|_{L^2} \left\|[\mathcal{R}, u]\theta\right\|_{\mathring{H}^{1/2}}.
\end{equation}
By (\ref{spe}) in Corollary \ref{bone},
\begin{equation}\label{goo2}
\left\|[\mathcal{R}, u]\theta\right\|_{\mathring{H}^{1/2}} \le C\, \|\omega\|_{L^2} \left(\|\theta_0\|_{L^{p_3}} +\|\theta_0\|_{L^{\frac{2}{1-\sigma}}}\right),
\end{equation}
where $p_3\ge \frac{2}{1/2-\sigma}$ is any constant. In addition,
\begin{equation}\label{goo3}
\|\omega\|_{L^2}  \le \|G\|_{L^2} + \|\mathcal{R}\theta\|_{L^2} \le \|G\|_{L^2} + \|\theta_0\|_{L^2}.
\end{equation}
Inserting \eqref{goo2} and \eqref{goo3} in \eqref{goo1} and applying Young's inequality, we obtain
$$
\frac{d}{dt}\|G\|_{L^2}^2 + \|\Lambda^{\frac12} G\|_{L^2}^2 \le C\, \|G\|_{L^2}^2 + C,
$$
where $C$'s are constants depending on the initial norm $\|\theta_0\|_{L^1\cap L^\infty}$. It then follows from Gronwall's inequality that, for any $t>0$,
$$
\|G(t)\|_{L^2}^2 + \int_0^t \|\Lambda^{\frac12} G\|_{L^2}^2 dt \le B(t),
$$
where $B(t)$ is an explicit smooth function of $t$.  The global bound for $\|\omega\|_{L^2}$ is then provided by (\ref{goo3}). This concludes the proof of Theorem \ref{globalw}.
\end{proof}

\vskip .1in
Finally we prove Proposition \ref{pp1}.

\begin{proof}[Proof of Proposition \ref{pp1}] By the definition of the Besov space $B^s_{p,q}$,
$$
\|[\mathcal{R},u] F\|^q_{B^s_{p,q}}  =\sum_{j=-1}^\infty 2^{qsj} \|\Delta_j [\mathcal{R},u] F \|_{L^p}^q.
$$
We decompose $\Delta_j [\mathcal{R},u] F$ into paraproducts,
$$
\Delta_j [\mathcal{R},u] F =I_1 + I_2 + I_3,
$$
where
\begin{eqnarray*}
I_1 &=& \sum_{|k-j|\le 2} \Delta_j (\mathcal{R} (S_{k-1} u \Delta_k F) - S_{k-1} u\,\mathcal{R} \Delta_k F),
\\
I_2 &=& \sum_{|k-j|\le 2} \Delta_j (\mathcal{R} (\Delta_k u \, S_{k-1} F)- \Delta_k u \mathcal{R } S_{k-1} F),
\\
I_3 &=& \sum_{k\ge j-1} \Delta_j (\mathcal{R} (\Delta_k  u \widetilde{\Delta}_k F) -\Delta_k  u \mathcal{R} \widetilde{\Delta}_k F).
\end{eqnarray*}
Here $\widetilde{\Delta}_k = \Delta_{k-1} + \Delta_k + \Delta_{k+1}$. For $k\ge 3$, the Fourier transform of $S_{k-1} u \Delta_k F$ is supported in the annulus $2^k \mathcal{A}$, where $\mathcal{A}$ denotes a fixed annulus. By Proposition 3.1 of \cite{HKR1}, there is a smooth function $h$ with compact support such that $\mathcal{R}$ acting on this term can be represented as a convolution with the kernel $h_k(x) \equiv 2^{dk} h(2^k x)$. More precisely,
$$
\mathcal{R} (S_{k-1} u \Delta_k F) - S_{k-1} u\,\mathcal{R} \Delta_k F = h_k*(S_{k-1} u \Delta_k F) - S_{k-1} u\, (h_k*\Delta_k F).
$$
Since
$$
\||x|^\delta 2^{dk} h(2^k x)\|_{L^1} \le C 2^{-\delta k},
$$
we apply Lemma \ref{newl} to obtain
$$
\left\|\mathcal{R} (S_{k-1} u \Delta_k F) - S_{k-1} u\,\mathcal{R} \Delta_k F\right\|_{L^p} \le C\, 2^{-\delta k} \|S_{k-1} u\|_{\mathring{B}^{\delta}_{p,\infty}} \|\Delta_k F\|_{L^\infty}.
$$
For $k<3$, we do not need the commutator structure and this difference can be directly estimated as follows. By the boundedness of $\mathcal{R}$ on $L^p$ for $p\in(1,\infty)$ and Bernstein's inequality (see Proposition \ref{bern} in the Appendix),
\begin{eqnarray*}
\left\|\mathcal{R} (S_{k-1} u \Delta_k F) - S_{k-1} u\,\mathcal{R} \Delta_k F\right\|_{L^p} &\le& C\, \|S_{k-1} u\,\Delta_k F\|_{L^p} \\
&\le& C\, \sum_{j=-1}^3 \|\Delta_j u\,\Delta_j F\|_{L^p}.
\end{eqnarray*}
Therefore,
\begin{eqnarray*}
\sum_{j=-1}^\infty 2^{qsj} \|I_1\|_{L^p}^q &\le& C\, \sum_{j\ge 3}^\infty 2^{(s-\delta)jq} \|S_{j-1} u\|^q_{\mathring{B}^{\delta}_{p,\infty}} \|\Delta_j F\|^q_{L^\infty} + C\, \sum_{j=-1}^3 \|\Delta_j u\,\Delta_j F\|^q_{L^p} \\
&\le& C\, \|u\|^q_{\mathring{B}^{\delta}_{p,\infty}}\, \|F\|^q_{B^{s-\delta}_{\infty,q}}
+ C\, \sum_{j=-1}^3 \|\Delta_j u\,\Delta_j F\|^q_{L^p}.
\end{eqnarray*}
The idea of bounding $I_2$ is similar. In fact, we have
$$
\sum_{j=-1}^\infty 2^{qsj} \|I_2\|_{L^p}^q \le C\,  \sum_{j\ge 3}^\infty 2^{(s-\delta)jq} \|\Delta_j u\|^q_{\mathring{B}^{\delta}_{p,\infty}} \|S_{j-1} F\|^q_{L^\infty} + C\, \sum_{j=-1}^3 \|\Delta_j u\,\Delta_j F\|^q_{L^p}.
$$
Furthermore,
\begin{eqnarray*}
&&\sum_{j\ge 3}^\infty 2^{(s-\delta)jq} \|\Delta_j u\|^q_{\mathring{B}^{\delta}_{p,\infty}} \|S_{j-1} F\|^q_{L^\infty} \\
&&\qquad \qquad  \le \|u\|^q_{\mathring{B}^{\delta}_{p,\infty}}\, \sum_{j=-1}^\infty 2^{(s-\delta)jq} \left[\sum_{m\le j-1} \|\Delta_m
F\|_{L^\infty}\right]^q \\
&&\qquad \qquad  \le \|u\|^q_{\mathring{B}^{\delta}_{p,\infty}}\, \sum_{j=-1}^\infty \left[\sum_{m\le j-1} 2^{(s-\delta)(j-m)} \, 2^{(s-\delta)m} \|\Delta_m F\|_{L^\infty}\right]^q\\
&&\qquad \qquad  \le C\, \|u\|^q_{\mathring{B}^{\delta}_{p,\infty}}\, \|F\|^q_{B^{s-\delta}_{\infty,q}}
\end{eqnarray*}
where we have used the fact that $s<\delta$ and the series inside the bracket can be viewed as a convolution of two other series.  The contribution from $I_3$ is
bounded by
$$
\sum_{j=-1}^\infty 2^{qsj} \|I_3\|_{L^p}^q \le C\,  \sum_{j \ge 3}^\infty 2^{sjq} \sum_{k\ge j-1} 2^{-\delta kq}\|\Delta_ku\|^q_{\mathring{B}^{\delta}_{p,\infty}} \|\widetilde{\Delta}_k F\|^q_{L^\infty} + C\, \sum_{j=-1}^3 \|\Delta_j u\,\Delta_j F\|^q_{L^p}.
$$
The first part can be further controlled by
\begin{eqnarray*}
&& \sum_{j\ge 3}^\infty 2^{sjq} \sum_{k\ge j-1} 2^{-\delta kq}\|\Delta_ku\|^q_{\mathring{B}^{\delta}_{p,\infty}} \|\widetilde{\Delta}_k F\|^q_{L^\infty}
\\
&& \qquad\qquad \le \|u\|^q_{\mathring{B}^{\delta}_{p,\infty}} \sum_{j=-1}^\infty \sum_{k\ge j-1}2^{s(j-k)q} 2^{(s-\delta)k q}\|\Delta_k F\|^q_{L^\infty}\\
&& \qquad\qquad \le \|u\|^q_{\mathring{B}^{\delta}_{p,\infty}}
\|F\|^q_{B^{s-\delta}_{\infty,q}}.
\end{eqnarray*}
We obtain (\ref{jj1}) by combining the estimates above. This completes the proof of Proposition \ref{pp1}.
\end{proof}

\vskip .4in
\section{Global bound for $\|\omega\|_{L^q}$ for $q>2$}
\label{exist}

This section establishes the global bounds stated in Theorem \ref{major1}. For the sake of clarity, this section is divided into four subsections. This first one provides a global bound for $\|\omega\|_{L^q}$ for $q\in (2, \frac{4}{2\sigma+1}]$. This bound holds for $0\le \gamma<\frac12$ and $\gamma\ge 0$. The second subsection proves the global bound for $G$ in the space-time norm $\widetilde{L}^rB^{s}_{q,1}$. This bound requires that $0\le \gamma<\frac14$ and $\gamma\ge 0$. The third subsection shows that, for $\sigma=0$ and any $\gamma\ge 0$, both $\omega$ and $\theta$ are bounded globally in $L^1_t B^0_{\infty,1}$. The final subsection presents the global $L^q$-bound for any $q\ge 2$ as long as $\sigma=0$ and $\gamma\ge 0$.

\vskip .1in
\subsection{Global bound for $\|\omega\|_{L^q}$ for $q\in (2, \frac{4}{2\sigma+1}]$}

This subsection proves a global bound for $\|\omega\|_{L^q}$ for $q\in (2, \frac{4}{2\sigma+1}]$. This result holds for any $0\le \sigma <\frac12$ and $\gamma\ge 0$. More precisely, we have the following theorem.

\begin{thm} \label{p1b}
Consider (\ref{GBou}) with $0\le \sigma <\frac12$ and $\gamma\ge 0$. Assume that $(\omega_0, \theta_0)$  satisfies the conditions in Theorem \ref{major1}, especially $(\omega_0, \theta_0) \in L^q$ for $q\in (2, \frac{4}{2\sigma+1}]$. Let $(\omega, \theta)$ be the corresponding solution of (\ref{GBou}). Then, for $ q\in (2, \frac{4}{2\sigma+1})$ with $\gamma>0$ and $q\in (2, \frac{4}{2\sigma+1}]$ with $\gamma=0$, and any $t>0$,
\begin{eqnarray}
\|\omega(t)\|_{L^q} &\le& B(t),\label{Oq}\\
\|G(t)\|_{L^q}^q + C\, \int_0^t \|G(\tau)\|_{L^{2q}}^q \,d\tau &\le& B(t),\label{Gq}
\end{eqnarray}
where $C$ is a constant depending on $q$ only and $B(t)$'s are smooth functions
of $t$.
\end{thm}

The following lemma, proven in \cite{HKR1}, will be used in the proof of Theorem \ref{p1b}.
\begin{lemma} \label{little}
Let $q\in [2,\infty)$ and $s\in (0,1)$. Then, for any smooth function $f$,
$$
\|f|f|^{q-2} \|_{\mathring{H}^s}  \le C \|f\|_{L^{2q}}^{q-2} \|f\|_{\mathring{H}^{s+1-\frac2{q}}}.
$$
\end{lemma}

\begin{proof}[Proof of Theorem \ref{p1b}] Multiplying (\ref{Geq}) by $G|G|^{q-2}$ and
integrating with respect to $x$ over $\mathbb{R}^2$, we obtain
$$
\frac1q \frac{d}{dt} \|G\|_{L^q}^q + \int G |G|^{q-2} \Lambda G \,dx = -\int G |G|^{q-2} \nabla\cdot [R, u] \theta\,dx.
$$
The dissipative part admits the lower bound
$$
\int G |G|^{q-2} \Lambda G \,dx \ge C \int |\Lambda^{\frac12} (|G|^{\frac{q}{2}})|^2 \ge C \|G\|_{L^{2q}}^q,
$$
where $C$ is a constant depending on $q$ only. When $\gamma=0$, we take
$$
s\ge \sigma, \quad q\in \left(2, \frac{4}{2\sigma+1}\right], \quad  s+1-\frac{2}{q} =\frac12.
$$
In the case when $\gamma>0$, we take $s>\sigma$. By H\"{o}lder's inequality,
$$
K=\left|\int G |G|^{q-2} \nabla\cdot [\mathcal{R}, u] \theta \right| \le \|G |G|^{q-2}\|_{\mathring{H}^s} \, \|[\mathcal{R}, u] \theta\|_{\mathring{H}^{1-s}}.
$$
By Lemma \ref{little},
$$
\|G |G|^{q-2}\|_{H^s} \le C \|G\|_{\mathring{H}^{s+1-\frac2{q}}} \|G\|_{L^{2q}}^{q-2} =C\, \|\Lambda^{\frac12} G\|_{L^2} \, \|G\|_{L^{2q}}^{q-2}.
$$
By Corollary \ref{bone}, for $1-s\le 1-\sigma$ or $s\ge \sigma$
$$
\|[\mathcal{R}, u] \theta\|_{H^{1-s}} \le C\,\|\omega\|_{L^2} (\|\theta\|_{L^{p_3}} +\|\theta\|_{L^{\frac{2}{1-\sigma}}}).
$$
Therefore, by Theorem \ref{globalw}.
$$
\|[\mathcal{R}, u] \theta\|_{H^{1-s}} \le C\, B(t).
$$
Combining the estimates above, we obtain
$$
\frac{d}{dt} \|G\|_{L^q}^q + C_q \|G\|_{L^{2q}}^q \le C\,B(t)\, \|\Lambda^{\frac12} G\|_{L^2} \, \|G\|_{L^{2q}}^{q-2}.
$$
Splitting the right-hand side by Young's inequality and using the bound in (\ref{Gbd}), we obtain (\ref{Gq}). (\ref{Oq}) follows from (\ref{Gq}) together with $\|\mathcal{R}\theta \|_{L^q}\le \|\theta_0\|_{L^q}$. This completes the proof of Theorem \ref{p1b}.
\end{proof}

\vskip .1in
\subsection{Global bound for $\|G\|_{\tilde{L}^r_t B^s_{q,1}}$}

This subsection presents a global bound on $G$ in the space-time space $\tilde{L}^r_t B^s_{q,1}$. The precise theorem can be stated as follows.

\begin{thm} \label{ltb}
Consider (\ref{GBou}) with $0\le \sigma <\frac14$ and $\gamma\ge 0$. Assume that $(\omega_0, \theta_0)$  satisfies the conditions in Theorem \ref{major1}. Let $(\omega, \theta)$ be the corresponding solution of (\ref{GBou}). Let $r$, $q$ and $s$ satisfy
$$
r\in [1,\infty], \quad s<1-\sigma, \quad \frac{2}{1-\sigma} < q < \frac{4}{1+2\sigma}.
$$
In the case when $\gamma=0$, we can take $q=4/(1+2\sigma)$. Then, for any $t>0$,
\begin{eqnarray} \label{ddd}
\|G\|_{\tilde{L}^r_t B^{s}_{q,1}} \le B(t).
\end{eqnarray}
\end{thm}
\begin{proof}
Let $j\ge -1$ be an integer. Applying $\Delta_j$ to \eqref{ggg} yields
$$
\partial_t \Delta_j G + u\cdot\nabla \Delta_j G + \Lambda \Delta_j G = -[\Delta_j, u\cdot \nabla] G - \Delta_j [\mathcal{R}, u\cdot\nabla]\theta.
$$
Taking the inner product with $\Delta_j G |\Delta_j G|^{q-2}$, we have
$$
\frac1q \frac{d}{dt} \|\Delta_j G\|_q^q + \int \Delta_j G |\Delta_j G|^{q-2} \Lambda \Delta_j G = J_1 + J_2,
$$
where
\begin{eqnarray*}
&& J_1 = - \int [\Delta_j, u\cdot \nabla] G \, \Delta_j G |\Delta_j G|^{q-2},\\
&& J_2 = - \int \Delta_j [\mathcal{R}, u\cdot\nabla]\theta \, \Delta_j G |\Delta_j G|^{q-2}.
\end{eqnarray*}
The dissipative part can be bounded below by
\begin{eqnarray*}
\int \Delta_j G |\Delta_j G|^{q-2}\, \Lambda \Delta_j G \ge C 2^j \|\Delta_j G\|_q^q,
\end{eqnarray*}
where $C$ is a  constant depending on $q$ only.  To estimate $J_1$, we write
\begin{eqnarray*}
[\Delta_j, u\cdot \nabla] G  =J_{11} + J_{12} + J_{13} + J_{14} + J_{15}
\end{eqnarray*}
with
\begin{eqnarray*}
&& J_{11} = \sum_{|j-k|\le 2} [\Delta_j, S_{k-1} u\cdot\nabla] \Delta_k G, \\
&& J_{12} = \sum_{|j-k|\le 2} (S_{k-1} u -S_j u)\cdot\nabla \Delta_j\Delta_k G, \\
&& J_{13} =  S_j u \cdot\nabla \Delta_j G, \\
&& J_{14} = \sum_{|j-k|\le 2} \Delta_j(\Delta_k u \cdot\nabla S_{k-1} G), \\
&& J_{15} = \sum_{k\ge j-1} \Delta_j(\Delta_k u \widetilde{\Delta}_k G).
\end{eqnarray*}
Since $\nabla\cdot u =0$, we have
$$
\int J_{13} |\Delta_j G|^{q-2} \Delta_j G =0.
$$
By H\"{o}lder's inequality,
$$
\left| \int J_{11} |\Delta_j G|^{q-2} \Delta_j G\right| \le \|J_{11}\|_{L^q} \|\Delta_j G\|_{L^q}^{q-1}.
$$
We write the commutator in terms of the integral,
$$
J_{11} = \int \Phi_j(x-y) \left(S_{k-1} u(y) -S_{k-1} u(x)\right) \cdot\nabla \Delta_k G(y) \,dy,
$$
where $\Phi_j$ is the kernel of the operator $\Delta_j$ and more details can be found in the Appendix. As in the proof of Lemma 3.3, we have, for any $\epsilon>0$,
\begin{eqnarray*}
\|J_{11}\|_{L^q} \le \||x|^{1-\sigma-\epsilon} \Psi_j(x)\|_{L^1}\, \|S_{k-1} u\|_{\mathring{B}^{1-\sigma-\epsilon}_{q,\infty}}\, \|\nabla \Delta_k G\|_{L^\infty}.
\end{eqnarray*}
Throughout the rest of this proof, $\epsilon>0$ is taken to be a small number such that
$$
\frac2q+ \sigma +\epsilon -1 <0.
$$
By the definition of $\Phi_j$ and Bernstein's inequality (see Appendix), we have
\begin{eqnarray*}
\|J_{11}\|_{L^q} &\le&  \||x|^{1-\sigma-\epsilon} \Psi_0(x)\|_{L^1} \, 2^{j(\sigma+\epsilon)} \|S_{j-1} u\|_{\mathring{B}^{1-\sigma-\epsilon}_{q,\infty}} \|\Delta_j G\|_{L^\infty} \\
&\le& C\, 2^{j(\sigma+\epsilon+\frac{2}{q})} \|\omega\|_{L^q} \, \|\Delta_j G\|_{L^q}.
\end{eqnarray*}
For $j\ge j_0$ with $j_0=2$,
\begin{eqnarray*}
\|J_{12}\|_{L^q} &\le& C \|\Delta_j u\|_{L^q} \|\nabla \Delta_j G\|_{L^\infty}\\
&\le& C 2^{j(\sigma+\epsilon+\frac{2}{q})} \, \|\omega\|_{L^q} \, \|\Delta_j G\|_{L^q}.
\end{eqnarray*}
Similarly, for $j\ge j_0$ with $j_0=2$,
\begin{eqnarray*}
\|J_{14}\|_{L^q} &\le& C \|\Delta_j u\|_{L^q} \|\nabla S_{j-1} G\|_{L^\infty}\\
&\le& C 2^{j(\sigma+\epsilon+\frac{2}{q})} \, \|\Lambda^{1-\sigma-\epsilon} \Delta_j u\|_{L^q}
\sum_{m\le j-2} 2^{(m-j)(1+\frac{2}{q})} \|\Delta_m G\|_{L^q}\\
&\le&  C 2^{j(\sigma+\epsilon+\frac{2}{q})} \,\|\omega\|_{L^q} \, \sum_{m\le j-2} 2^{(m-j)(1+\frac{2}{q})} \|\Delta_m G\|_{L^q}.
\end{eqnarray*}
$J_{15}$ be bounded by
\begin{eqnarray*}
\|J_{15}\|_{L^q} &\le&  C 2^{j(\sigma+\epsilon+\frac{2}{q})} \, \sum_{k\ge j-1} \|\Lambda^{1-\sigma-\epsilon} \Delta_k u\|_{L^q} 2^{(j-k)(1-\sigma-\epsilon-\frac{2}{q})} \, \|\Delta_k G\|_{L^2}\\
&\le&  C 2^{j(\sigma+\epsilon+\frac{2}{q})} \,\|\omega\|_{L^q} \, \sum_{k\ge j-1} 2^{(j-k)(1-\sigma-\epsilon-\frac{2}{q})} \, \|\Delta_k G\|_{L^2}.
\end{eqnarray*}
Thus, we have obtained that
\begin{eqnarray*}
\|J_1\|_{L^q} &\le&  C\,2^{j(\sigma+\epsilon+\frac{2}{q})} \, \|\omega\|_{L^q}  \Big[\|\Delta_j G\|_{L^2} + \sum_{m\le j-2} 2^{(m-j)(1+\frac{2}{q})} \|\Delta_m G\|_{L^q} \\
&& + \sum_{k\ge j-1} 2^{(j-k)(1-\sigma-\epsilon-\frac{2}{q})} \, \|\Delta_k G\|_{L^2}\Big].
\end{eqnarray*}
For $0\le s <1 + \frac2{q}$, we have
\begin{eqnarray*}
\sum_{m\le j-2} 2^{(m-j)(1+\frac{2}{q})} \|\Delta_m G\|_{L^q} &=& 2^{-js} \sum_{m\le j-2} 2^{(m-j)(1+\frac{2}{q}-s)} 2^{ms} \|\Delta_m G\|_{L^q} \\
&=& 2^{-js}\, \|G\|_{B^s_{q,1}}.
\end{eqnarray*}
Similarly,
\begin{eqnarray*}
\sum_{k\ge j-1} 2^{(j-k)(1-\sigma-\epsilon-\frac{2}{q})} \, \|\Delta_k G\|_{L^2} \le 2^{-js}\, \|G\|_{B^s_{q,1}}.
\end{eqnarray*}
Therefore,
$$
\|J_1\|_{L^q} \le C\,2^{j(\sigma+\epsilon+\frac{2}{q})} \, \|\omega\|_{L^q}  \Big[\|\Delta_j G\|_{L^2} + 2^{-js}\, \|G\|_{B^s_{q,1}}\Big].
$$
By H\"{o}lder's inequality and an argument as in the proof of Proposition \ref{pp1},
\begin{eqnarray*}
|J_2| &\le& \|\Delta_j [R, u\cdot\nabla]\theta\|_{L^q} \,
\|\Delta_j G\|_{L^q}^{q-1} \\
&\le& C\, 2^{j(\sigma+\epsilon)} \|\omega\|_{L^q} \, \|\Delta_j\theta\|_{L^\infty} \, \|\Delta_j G\|_{L^q}^{q-1}.
\end{eqnarray*}
Collecting the estimates, we have
\begin{eqnarray*}
&& \frac{d}{dt} \|\Delta_j G\|_{L^q} + C\, 2^j  \|\Delta_j G\|_q  \le  C\, 2^{j(\sigma+\epsilon)} \|\omega\|_{L^q} \, \|\theta_0\|_{L^\infty} \\
&& \qquad\quad +\, C\,2^{j(\sigma+\epsilon+\frac{2}{q})} \, \|\omega\|_{L^q}  \Big[\|\Delta_j G\|_{L^2} + 2^{-js}\, \|G\|_{B^s_{q,1}}\Big].
\end{eqnarray*}
Integrating in time and using the fact that $\|\omega\|_{L^q} \le B(t)$, we have
\begin{eqnarray*}
\|\Delta_j G(t)\|_{L^q} &\le& e^{-2^j t}\|\Delta_j G(0)\|_{L^q} +  C\, 2^{j(\sigma+\epsilon-1)} \|\theta_0\| \, B(t) \\
&& + C\,2^{j(\sigma+\epsilon+\frac{2}{q})} \,B(t) \int_0^t e^{-2^j(t-s)}\Big[\|\Delta_j G\|_{L^2} + 2^{-js}\, \|G\|_{B^s_{q,1}}\Big]\,ds.
\end{eqnarray*}
Taking $L^r$-norm in time and applying Young's inequality, we obtain
\begin{eqnarray*}
\|\Delta_j G\|_{L^r_t L^q} &\le& C\, 2^{-\frac1r j} \|\Delta_j G(0)\|_{L^q} + C\, 2^{j(-1+ \sigma+\epsilon)} \|\theta_0\|_{L^\infty}\, B(t) \\
&& + C\, 2^{j(-1+\sigma+\epsilon +\frac{2}{q})} B(t) \Big[\|\Delta_j G\|_{L^r_t L^q} + 2^{-js}\, \|G\|_{\widetilde{L}^r_t B^s_{q,1}}\Big].
\end{eqnarray*}
Multiplying $2^{js}$, summing over $j\ge -1$ and using the fact $s<1-\sigma$, we obtain
\begin{eqnarray} \label{Gm}
\|G\|_{\widetilde{L}^r B^s_{q,1}} &\le& C\,\|G(0)\|_{B^{s-\frac1r}_{q,1}} + C\, \|\theta_0\|_{L^\infty} B_1(t) + K,
\end{eqnarray}
where
$$
K = C\, \sum_{j\ge -1} 2^{j(-1+\sigma+\epsilon +\frac{2}{q})} B(t) \Big[2^{js}\,\|\Delta_j G\|_{L^r_t L^q} +  \|G\|_{\widetilde{L}^r_t B^s_{q,1}}\Big].
$$
We choose $N$ such that
$$
C\, 2^{N(-1+\sigma+\epsilon +\frac{2}{q})} B(t)  \le \frac14
$$
and decompose the sum in $K$ into two parts: $j\le N$ and $j>N$. Using the fact that $\|G\|_{L^q}$ is bounded, the sum for $j\le N$ can be bounded by $B(t) 2^{s N}$ for a smooth function $B(t)$. The sum for $j>N$ is bounded by $\frac12 \|G\|_{\widetilde{L}^r B^s_{q,1}}$. That is,
\begin{equation}\label{ke}
K \le  B(t) 2^{s N} + \frac12 \|G\|_{\widetilde{L}^r B^s_{q,1}}.
\end{equation}
Inserting (\ref{ke}) in (\ref{Gm}) yields
\eqref{ddd}. This completes the proof of Theorem \ref{ltb}.
\end{proof}

\subsection{Bounds for $\|\omega\|_{B^{0,\gamma}_{\infty,1}}$ and $\|\theta\|_{B^{0,\gamma}_{\infty,1}}$}

This subsection provides global bounds for $\|\omega\|_{B^{0,\gamma}_{\infty,1}}$ and $\|\theta\|_{B^{0,\gamma}_{\infty,1}}$.

\begin{thm} \label{tim}
Consider (\ref{GBou}) with $\sigma=0$ and $\gamma\ge 0$. Assume that $(\omega_0, \theta_0)$  satisfies the conditions in Theorem \ref{major1}, especially $(\omega_0, \theta_0) \in B^{0,\gamma}_{\infty,1}$. Let $(\omega, \theta)$ be the corresponding solution of (\ref{GBou}). Then, for any $t>0$,
\begin{equation} \label{b4b}
\|\omega\|_{L^1_tB^{0,\gamma}_{\infty,1}} \le B(t), \quad \|\theta\|_{L^1_t B^{0,\gamma}_{\infty,1}} \le B(t).
\end{equation}
\end{thm}

\begin{proof}
Taking $r=1$ and $\frac2q<s<1-\sigma$, we obtain from Theorem \ref{ltb} that
$$
\|G\|_{L^1_t B^s_{q,1}} \le B(t).
$$
This bound especially imply that
$$
\|G\|_{L^1_t B^{0,\gamma}_{\infty,1}} \le B(t).
$$
In fact, by Bernstein's inequality,
$$
\|G\|_{ B^{0,\gamma}_{\infty,1}}  = \sum_{j\ge -1} (1+|j|)^\gamma \|\Delta_j G\|_{L^\infty} \le \sum_{j\ge -1} (1+|j|)^\gamma 2^{\frac{2}{q} j} \|\Delta_j G\|_{L^q} \le C \|G\|_{B^s_{q,1}}.
$$
Since $G=\omega- R\theta$,
$$
\|\omega\|_{B^{0,\gamma}_{\infty,1}} \le \|G\|_{B^{0,\gamma}_{\infty,1}} + \|\mathcal{R}\theta\|_{B^{0,\gamma}_{\infty,1}}.
$$
In addition,
$$
\|\mathcal{R}\theta\|_{B^{0,\gamma}_{\infty,1}} \le \|\Delta_{-1}\theta\|_{L^\infty} + \|\theta\|_{B^{0,\gamma}_{\infty,1}} \le \|\theta_0\|_{L^2} + \|\theta\|_{B^{0,\gamma}_{\infty,1}}
$$
By Lemma \ref{ggd} below and
$$
\|\nabla u\|_{L^\infty} \le \|\omega\|_{L^2} + \|\omega\|_{B^{0,\gamma}_{\infty,1}},
$$
we obtain
$$
\|\theta\|_{B^{0,\gamma}_{\infty,1}}  \le \|\theta_0\|_{B^{0,\gamma}_{\infty,1}} \left(1+ \int_0^t \|\omega\|_{L^2}\,dt\right) +  \|\theta_0\|_{B^{0,\gamma}_{\infty,1}}  \int_0^t \|\omega\|_{B^{0,\gamma}_{\infty,1}}\,dt.
$$
Therefore, we have obtained
\begin{eqnarray*}
\|\omega\|_{B^{0,\gamma}_{\infty,1}} &\le& \|G\|_{B^{0,\gamma}_{\infty,1}} + \|\theta_0\|_{L^2} + \|\theta_0\|_{B^{0,\gamma}_{\infty,1}} \left(1+ \int_0^t \|\omega\|_{L^2}\,dt\right)\\
&& + \,\|\theta_0\|_{B^{0,\gamma}_{\infty,1}}  \int_0^t \|\omega\|_{B^{0,\gamma}_{\infty,1}}\,dt.
\end{eqnarray*}
If we set $Z(t) = \|\omega\|_{L^1_tB^{0,\gamma}_{\infty,1}}$, then
$$
Z(t) \le B(t) + \|\theta_0\|_{B^{0,\gamma}_{\infty,1}} \int_0^t Z(\tau)\,d\tau.
$$
(\ref{b4b}) then follows from Gronwall's inequality.
\end{proof}

\vskip .1in
The following lemma has been used in the proof of Theorem \ref{tim}.
\begin{lemma} \label{ggd}
Let $\theta$ satisfy
$$
\partial_t \theta + u\cdot\nabla\theta +  \Lambda \theta =f.
$$
Let $\gamma\ge 0$ and $\rho\in [1,\infty]$. Then, for any $t>0$,
$$
\|\theta(t)\|_{B^{0,\gamma}_{\rho,1}} \le \left(\|\theta_0\|_{B^{0,\gamma}_{\rho,1}} + \|f\|_{L^1_tB^{0,\gamma}_{\rho,1}}\right) \left(1+\int_0^t \|\nabla u\|_{L^\infty} \,dt \right).
$$
\end{lemma}

\begin{proof} Theorem 4.5 of \cite[p.432]{HKR2} states a similar result for the Besov space $B^0_{\rho,1}$. The generalization to the Besov space $B^{0,\gamma}_{\rho,1}$ presented here is not completely trivial. For an integer $k\ge -1$, consider the solution $\theta_k$ of the equation $$
\partial_t \theta_k + u\cdot\nabla\theta_k + \Lambda \theta_k =\Delta_k f, \quad \theta(x,0)= \Delta_k \theta_0.
$$
For any $s\in (-1, 1)$ and $\rho\in [1,\infty]$, we have the standard Besov estimate
\begin{equation}\label{bjjj}
\|\theta_k\|_{B^{s}_{\rho,\infty}} \le C \left(\|\Delta_k \theta_0\|_{B^{s}_{\rho,\infty}} + \|\Delta_kf\|_{L^1_tB^{s}_{\rho,\infty}}\right) e^{CV(t)},
\end{equation}
where $V(t) = \|\nabla u\|_{L^1_t L^\infty}$. Setting $s=\pm \frac12$ in \eqref{bjjj}, we obtain
\begin{equation}\label{bjj}
\|\Delta_j \theta_k\|_{L^\rho} \le C 2^{-\frac12|j-k|} \left(\|\Delta_k \theta_0\|_{L^\rho} + \|\Delta_kf\|_{L^1_t L^\rho}\right) e^{C V(t)}.
\end{equation}
 Clearly $\theta =\sum \Delta_k \theta$ and thus
\begin{eqnarray*}
\|\theta(t)\|_{B^{0,\gamma}_{\rho,1}} = \sum_{j=-1}^\infty (1+|j|)^\gamma \|\Delta_j \theta\|_{L^\rho}
\le \sum_{j=-1}^\infty \sum_{k=-1}^\infty (1+|j|)^\gamma \|\Delta_j \theta_k\|_{L^\rho}.
\end{eqnarray*}
For an integer $N$ to be fixed later, we decompose the double summation in the inequality above into two parts: $J_1$ for $|j-k|\ge N$ and $J_2$ for $|j-k| <N$.
Invoking \eqref{bjj}, we have
\begin{eqnarray*}
J_1 &\le& C\, \sum_{j=-1}^\infty\sum_{|j-k|\ge N} (1+|j|)^\gamma 2^{-\frac12|j-k|} \left(\|\Delta_k \theta_0\|_{L^\rho} + \|\Delta_k f\|_{L^1_t L^\rho}\right) e^{C V(t)} \\
&=& C\, \sum_{j=-1}^\infty\sum_{|j-k|\ge N} \frac{(1+|j|)^\gamma}{(1+|k|)^\gamma} 2^{-\frac12|j-k|} (1+|k|)^\gamma\left(\|\Delta_k \theta_0\|_{L^\rho} + \|\Delta_k f\|_{L^1_t L^\rho}\right) e^{C V(t)}.
\end{eqnarray*}
In the summation above, in the case when $k\ge j+N$, we certainly have $(1+|j|)^\gamma/(1+|k|)^\gamma \le 1$. In the case when $j\ge k+N$, we have
$$
\frac{(1+|j|)^\gamma}{(1+|k|)^\gamma} 2^{-\delta (j-k)} \le C
$$
for any fixed $\delta>0$, where $C$ is independent of $j$ and $k$.  Therefore, for $0<\delta<\frac12$, we have
$$
J_1 \le 2^{-(\frac12-\delta) N} e^{C V(t)} \left( \|\theta_0\|_{B^{0,\gamma}_{\rho,1}} + \|f\|_{L^1_t B^{0,\gamma}_{\rho,1}} \right).
$$
To bound $J_2$, we handle $\|\Delta_j \theta_k\|_{L^\rho}$ differently. Through a standard $L^\rho$-estimate,
$$
\|\Delta_j \theta_k\|_{L^\rho} \le \|\theta_k\|_{L^\rho} \le \|\Delta_k \theta_0\|_{L^\rho} + \|\Delta_kf\|_{L^1_t L^\rho}.
$$
Therefore,
$$
J_2 = \sum_{j=-1}^\infty \sum_{|j-k|< N} (1+|j|)^\gamma \left( \|\Delta_k \theta_0\|_{L^\rho} + \|\Delta_kf\|_{L^1_t L^\rho}\right).
$$
We can show that, for each $k$ satisfying $|j-k| <N$,
\begin{eqnarray*}
&& \sum_{j=-1}^\infty (1+|j|)^\gamma \left( \|\Delta_k \theta_0\|_{L^\rho} + \|\Delta_kf\|_{L^1_t L^\rho}\right)
\\&& \qquad \qquad\le C \left( \|\theta_0\|_{B^{0,\gamma}_{\rho,1}} + \|f\|_{L^1_t B^{0,\gamma}_{\rho,1}} \right) + \sup_{j\ge -1} (1+|j|)^{1+\gamma} \left( \|\Delta_k \theta_0\|_{L^\rho} + \|\Delta_kf\|_{L^1_t L^\rho}\right) \\
\\&& \qquad \qquad\le C \left( \|\theta_0\|_{B^{0,\gamma}_{\rho,1}} + \|f\|_{L^1_t B^{0,\gamma}_{\rho,1}} \right).
\end{eqnarray*}
where $C$'s are constants independent of $N$. In the last inequality we have used the fact that $B^{0,\gamma}_{\rho,1} \hookrightarrow B^{0,1+\gamma}_{\rho,\infty}$. The inequality above can be established by writing $j=k+m$ with $0<m<N$ and split the summation for $j$ into two parts: one part for $j\le m$ and the other for $j> m$. We omit further details. Therefore,
$$
J_2 \le C\, N \left( \|\theta_0\|_{B^{0,\gamma}_{\rho,1}} + \|f\|_{L^1_t B^{0,\gamma}_{\rho,1}} \right).
$$
Consequently,
\begin{eqnarray*}
\|\theta(t)\|_{B^{0,\gamma}_{\rho,1}} &\le& J_1 +J_2 \\
&\le& 2^{-(\frac12-\delta) N} e^{C V(t)} \left( \|\theta_0\|_{B^{0,\gamma}_{\rho,1}} + \|f\|_{L^1_t B^{0,\gamma}_{\rho,1}} \right) + C\, N \left( \|\theta_0\|_{B^{0,\gamma}_{\rho,1}} + \|f\|_{L^1_t B^{0,\gamma}_{\rho,1}} \right).
\end{eqnarray*}
The desired inequality follows by taking $N$ such that $2^{-(\frac12-\delta) N} e^{C V(t)}$ is of order $1$. This completes the proof of Lemma \ref{ggd}.
\end{proof}

\vskip .1in
\subsection{Global bound $\|\omega\|_{L^q}$ for any $q >2$}

The goal of this subsection is to establish a global bound for $\|\omega\|_{L^q}$ for any $q >2$.

\begin{thm}
Consider (\ref{GBou}) with $\sigma=0$ and $\gamma\ge 0$.
Assume $(\omega_0, \theta_0)$ satisfies the conditions stated in Theorem \ref{major1}. Let $(\omega, \theta)$ be the corresponding solution. Then, for any $q\ge 2$,
\begin{equation}\label{what}
\|\omega(t)\|_{L^q} \le B(t).
\end{equation}
\end{thm}

\begin{proof}
It is clear from \eqref{ggg} that, for any $q\ge2$,
$$
\|G\|_{L^q} \le \|G_0\|_{L^q} + \int_0^t \|[\mathcal{R}, u\cdot\nabla \theta]\|_{L^q} \,dt.
$$
According to the commutator estimate of Proposition \ref{clog} below,
$$
\|G\|_{L^q} \le \|G_0\|_{L^q} + \int_0^t \|\omega(s)\|_{L^q}\, \|\theta(s)\|_{B^{0,\gamma}_{\infty,1}} \,ds.
$$
Therefore,
$$
\|\omega(t)\|_{L^q} \le \|\theta_0\|_{L^q} + \|G_0\|_{L^q} + \int_0^t \|\omega(s)\|_{L^q}\, \|\theta(s)\|_{B^{0,\gamma}_{\infty,1}} \,ds.
$$
Gronwall's inequality combined with the bound in Theorem \ref{b4b} yields (\ref{what}). \end{proof}

\vskip .1in
\begin{prop} \label{clog}
Let $\gamma\ge 0$. Assume that $u$ and $\omega$ are related by
$$
u = \nabla^\perp \psi, \quad \Delta\psi
=(\log(I-\Delta))^\gamma \omega.
$$
Then, for any $q \ge 2$, we have
$$
\|[\mathcal{R}, u\cdot\nabla]\theta\|_{B^0_{q,1}} \le C \|\omega\|_{L^q} \|\theta\|_{B^{0,\gamma}_{\infty,1}}.
$$
\end{prop}

\begin{proof} For any integer $j\ge -1$, we write
$$
\Delta_j [\mathcal{R}, u\cdot\nabla]\theta = J_1 + J_2 +J_3,
$$
where
\begin{eqnarray*}
J_1 &=& \sum_{|k-j|\le 2} \Delta_j (\mathcal{R} S_{k-1}u \cdot\nabla \Delta_k \theta) -\Delta_j (S_{k-1}u \cdot\nabla \mathcal{R} \Delta_k \theta), \\
J_2 &=& \sum_{|k-j|\le 2} \Delta_j (\mathcal{R} (\Delta_k u \,\cdot\nabla  S_{k-1} \theta)- \Delta_k u \cdot \nabla \mathcal{R} S_{k-1} \theta),
\\
J_3 &=& \sum_{k\ge j-1} \Delta_j (\mathcal{R} (\Delta_k  u \cdot \nabla\widetilde{\Delta}_k \theta) -\Delta_k  u \mathcal{R} \cdot \nabla\widetilde{\Delta}_k \theta).
\end{eqnarray*}
Estimating these terms in a similar fashion as in the proof of Proposition \ref{pp1}, we have $$
\|J_1\|_{L^q} \le C\, \|S_{j-1}\omega\|_{L^q} (1+|j|)^\gamma \|\Delta_j\theta\|_{L^\infty}.
$$
For $j\ge j_0$ with $j_0=2$, we have
\begin{eqnarray*}
\|J_2\|_{L^q} &\le&  C\, \|\omega\|_{L^q} 2^{-j} (1+|j|)^\gamma \|S_{j-1} \nabla \theta\|_{L^\infty} \\
&\le& C\, \|\omega\|_{L^q} \sum_{m\le j-1} \frac{2^m (1+|j|)^\gamma}{2^j (1+|m|)^\gamma} (1+|m|)^\gamma\|\Delta_m\theta\|_{L^\infty}
\end{eqnarray*}
and
\begin{eqnarray*}
\|J_3\|_{L^q} &\le&  C\,  \|\omega\|_{L^q} \sum_{k\ge j-1} 2^{(j-k)} (1+|k|)^\gamma
\|\Delta_k \theta\|_{L^\infty}.
\end{eqnarray*}
Therefore,
$$
\|[\mathcal{R}, u\cdot\nabla]\theta\|_{B^0_{q,1}} \le \sum_{j\ge -1}  \|\Delta_j[\mathcal{R}, u\cdot\nabla]\theta\|_{L^q} \le C \|\omega\|_{L^q} \|\theta\|_{B^{0,\gamma}_{\infty,1}}.
$$
This completes the proof of Proposition \ref{clog}.
\end{proof}

\vskip .4in
\section{Uniqueness}
\label{uniq}

This section proves the uniqueness part of Theorem \ref{major1}. For the sake of clarity, we state it as a theorem.

\begin{thm}\label{major1uni}
Assume that $(\omega_0, \theta_0)$ satisfies the conditions stated in Theorem \ref{major1}. Let $\sigma=0$, $\gamma\ge 0$ and $q>2$. Let $(\omega^{(1)}, \theta^{(1)})$ and $(\omega^{(2)}, \theta^{(2)})$ be two solutions of \eqref{GBou} satisfying, for any $t>0$,
$$
\omega^{(1)}, \, \omega^{(2)} \in L^2\cap L^q \cap L^1_t B^{0,\gamma}_{\infty,1}, \quad \theta^{(1)}, \theta^{(2)}\in L^2\cap L^\infty \cap L^1_tB^{0,\gamma}_{\infty,1}.
$$
Then they must coincide.
\end{thm}

\begin{proof}[Proof of Theorem \ref{major1uni}]

Let $u^{(1)}$ and $u^{(2)}$ be the corresponding velocity fields, namely
$$
u^{(j)} =\nabla^\perp \psi^{(j)}, \quad \Delta \psi^{(j)} =  (\log(I-\Delta))^\gamma \omega^{(j)}, \quad j=1,2.
$$
Let $v^{(j)}= (\log(I-\Delta))^\gamma u^{(j)}$, $j=1,2$. Then the differences
$$
u = u^{(1)}- u^{(2)}, \qquad \theta = \theta^{(1)}-\theta^{(2)}, \quad v =v^{(1)} -v^{(2)}, \qquad p= p^{(1)}-p^{(2)}
$$
satisfy
\begin{eqnarray*} \label{difeq}
&& \partial_t v + u^{(2)}\cdot\nabla v + u\cdot \nabla v^{(1)} - \sum_{j=1}^2 \left(u_j^{(2)} \nabla v_j + u_j \nabla v_j^{(1)}\right) + \Lambda v = -\nabla p + \theta e_2, \\
&& \partial_t \theta + u\cdot\nabla \theta^{(1)} + u^{(2)} \cdot\nabla \theta =0.
\end{eqnarray*}
By Lemmas \ref{ggo1} and \ref{ggo2} below,
\begin{eqnarray*}
\|\theta(t)\|_{B^{-1}_{2,\infty}}&\le& \|\theta(0)\|_{B^{-1}_{2,\infty}} + \, C\, \int_0^t \|v(s)\|_{L^2}\, \|\theta^{(1)}(s)\|_{B^{0,\gamma}_{\infty,1}} \,ds\\
&&  + \,C\,\int_0^t \|\omega^{(2)}(s)\|_{B^{0,\gamma}_{\infty,1}}  \|\theta(s)\|_{B^{-1}_{2,\infty}}\,ds,
\end{eqnarray*}
\begin{eqnarray*}
\|v(t)\|_{B^0_{2,\infty}} &\le& \|v(0)\|_{B^0_{2,\infty}} +\|\theta(s)\|_{B^{-1}_{2,\infty}} \\
&& + \,C\, \int_0^t \|v(s)\|_{L^2} \, \left(\|\omega^{(1)}(s)\|_{B^{0,\gamma}_{\infty,1}} +\|\omega^{(2)}(s)\|_{B^{0,\gamma}_{\infty,1}}\right)\,ds.
\end{eqnarray*}
To further the estimate, we bound $\|v\|_{L^2}$ in terms of $\|v\|_{B^0_{2,\infty}}$ by the interpolation inequality (see Lemma 6.11 of \cite[p.2173]{HKR1})
$$
\|v\|_{L^2} \le C\,\|v\|_{B^0_{2,\infty}} \log\left(1+ \frac{\|v\|_{H^1}}{\|v\|_{B^0_{2,\infty}}}\right)
$$
and use the fact that $\|v\|_{H^1} \le \|\omega^{(1)}\|_{L^2} + \|\omega^{(2)}\|_{L^2}$. Combining the inequalities above and setting
$$
Y(t) = \|\theta(t)\|_{B^{-1}_{2,\infty}} + \|v(t)\|_{B^0_{2,\infty}},
$$
we obtain
\begin{eqnarray*}
Y(t) \le 2\, Y(0) + C\, \int_0^t D_1(s) Y(s) \, \log(1+ D_2(s)/Y(s)) \,ds
\end{eqnarray*}
where $
D_1(s) = \|\theta^{(1)}(s)\|_{B^{0,\gamma}_{\infty,1}} + \|\omega^{(1)}(s)\|_{B^{0,\gamma}_{\infty,1}} +\|\omega^{(2)}(s)\|_{B^{0,\gamma}_{\infty,1}},
$ and $D_2(s)= \|\omega^{(1)}(s)\|_{L^2} + \|\omega^{(2)}(s)\|_{L^2}$. Since $D_1$ and $D_2$ are integrable, we obtain by Osgood's inequality that $Y(t)\equiv 0$. A statement of Osgood's theorem is provided in the Appendix. This completes the proof of Theorem \ref{major1uni}.
\end{proof}

\vskip .1in
\begin{lemma} \label{ggo1}
Assume that $\theta$ satisfies
\begin{equation} \label{theq}
\partial_t \theta + u\cdot\nabla \theta^{(1)} + u^{(2)} \cdot\nabla \theta =0,
\end{equation}
where $u$, $\theta^{(1)}$ and $u^{(2)}$ are as defined in the proof of Theorem \ref{major1uni}. Then, for any $t>0$,
\begin{eqnarray} \label{ggi}
\|\theta(t)\|_{B^{-1}_{2,\infty}}&\le& \|\theta(0)\|_{B^{-1}_{2,\infty}} + \, C\, \int_0^t \|v(s)\|_{L^2}\, \|\theta^{(1)}(s)\|_{B^{0,\gamma}_{\infty,1}} \,ds\nonumber\\
&&  + \,C\,\int_0^t \|\omega^{(2)}(s)\|_{B^{0,\gamma}_{\infty,1}}  \|\theta(s)\|_{B^{-1}_{2,\infty}}\,ds.
\end{eqnarray}
where $v$ is as defined in the proof of Theorem \ref{major1uni}.
\end{lemma}

\begin{proof}
Let $j\ge -1$. Applying $\Delta_j$ to (\ref{theq}), taking the inner product of $\Delta_j \theta$ with the resulting equation and applying H\"{o}lder's inequality, we obtain
\begin{eqnarray} \label{mmm}
\frac12 \frac{d}{dt} \|\Delta_j \theta\|^2_{L^2} \le \|\Delta_j (u\cdot\nabla \theta^{(1)}\|_{L^2} \|\Delta_j \theta\|_{L^2} - \int \Delta_j \theta \Delta_j (u^{(2)}\cdot\nabla \theta)\,dx.
\end{eqnarray}
To estimate the first term, we write
\begin{equation}\label{dd00}
\Delta_j (u\cdot\nabla \theta^{(1)}) = J_1 + J_2 + J_3,
\end{equation}
where $J_1$, $J_2$ and $J_3$ are given by
\begin{eqnarray*}
&& J_1 = \sum_{|j-k|\le 2} \Delta_j (S_{k-1}u\cdot\nabla \Delta_k \theta^{(1)}), \\
&& J_2 = \sum_{|j-k|\le 2} \Delta_j (\Delta_k u\cdot\nabla  S_{k-1}\theta^{(1)}), \\
&& J_3 = \sum_{k\ge j-1} \Delta_j (\Delta_k u\cdot\nabla \widetilde{\Delta}_k\theta^{(1)}).
\end{eqnarray*}
$J_1$, $J_2$ and $J_3$ can be estimated as follows.
\begin{eqnarray*}
\|J_1\|_{L^2} &\le& C\, 2^j \|S_{j-1} u\|_{L^2} \|\Delta_j \theta^{(1)}\|_{L^\infty} \\
&\le& C\, 2^j \|v\|_{L^2} (1+|j|)^\gamma \|\Delta_j \theta^{(1)}\|_{L^\infty} \\
&\le& C\, 2^j \,\|v\|_{L^2} \,\|\theta^{(1)}\|_{B^{0, \gamma}_{\infty, \infty}}.
\end{eqnarray*}
\begin{eqnarray*}
\|J_2\|_{L^2} &\le& C\, \|\Delta_j u\|_{L^2} \|S_{j-1} \nabla \theta\|_{L^\infty} \\
&\le& C\, \|\Delta_j v\|_{L^2} (1+|j|)^\gamma \sum_{m\le j-2} 2^m  \|\Delta_m \theta^{(1)}\|_{L^\infty} \\
&\le& C\, 2^j \,\|v\|_{B^0_{2,\infty}} \sum_{m\le j-2} \frac{2^m (1+|m|)^{-\gamma}}{2^j (1+|j|)^{-\gamma}}(1+|m|)^\gamma \|\Delta_m \theta^{(1)}\|_{L^\infty} \\
&\le& C\, 2^j \,\|v\|_{B^0_{2,\infty}} \sum_{m\le j-2} (1+|m|)^\gamma \|\Delta_m \theta^{(1)}\|_{L^\infty}\\
&\le& C\, 2^j \,\|v\|_{B^0_{2,\infty}} \, \|\theta^{(1)}\|_{B^{0,\gamma}_{\infty,1}}.
\end{eqnarray*}
\begin{eqnarray*}
\|J_3\|_{L^2} &\le& C\, 2^j \sum_{k\ge j-1} (1+|k|)^\gamma \|\Delta_k v\|_{L^2} \, \|\Delta_k \theta^{(1)}\|_{L^\infty} \\
&\le& C\, 2^j \,\|v\|_{B^0_{2,\infty}} \, \|\theta^{(1)}\|_{B^{0,\gamma}_{\infty,1}}.
\end{eqnarray*}
To estimate the second term in (\ref{mmm}), we write
\begin{equation}\label{ddee}
\Delta_j (u^{(2)}\cdot\nabla \theta) =K_1 + K_2 + K_3 + K_4 + K_5,
\end{equation}
where
\begin{eqnarray*}
&& K_1 = \sum_{|j-k|\le 2} [\Delta_j, S_{k-1}u^{(2)}\cdot\nabla] \Delta_k \theta, \\ && K_2 = \sum_{|j-k|\le 2} (S_{k-1}u^{(2)} -S_j u^{(2)})\cdot\nabla \Delta_j \Delta_k\theta, \\
&& K_3= S_j  u^{(2)} \cdot\nabla \Delta_j \theta,\\
&& K_4= \sum_{|j-k|\le 2} \Delta_j ( \Delta_k u^{(2)} \cdot\nabla S_{k-1}\theta),\\
&& K_5 = \sum_{k\ge j-1} \Delta_j (\Delta_k u^{(2)}\cdot\nabla \widetilde{\Delta}_k\theta).
\end{eqnarray*}
Correspondingly the second term in (\ref{mmm}) can be decomposed into five integrals. Since $\nabla\cdot u^{(2)} =0$,
$$
\int \Delta_j \theta  K_3 \,dx =0.
$$
Therefore, by H\"{o}lder's inequality,
$$
\left| \int \Delta_j \theta \Delta_j (u^{(2)}\cdot\nabla \theta)\,dx \right| \le \|\Delta_j \theta\|_{L^2} \left(\|K_1\|_{L^2} + \|K_2\|_{L^2} + \|K_4\|_{L^2} + \|K_5\|_{L^2} \right).
$$
By a standard commutator estimate,
\begin{eqnarray*}
\|K_1\|_{L^2} &\le& C\, \|x \Phi_j(x)\|_{L^1} \, \|\nabla S_{j-1} u^{(2)}\|_{L^\infty} \|\nabla \Delta_j \theta\|_{L^2} \\
&\le& C\, \|x \Phi_0(x)\|_{L^1} \,\|\omega^{(2)}\|_{B^{0,\gamma}_{\infty,1}} \| \Delta_j \theta\|_{L^2}.
\end{eqnarray*}
For $j\ge j_0$ with $j_0= 2$, we apply Berstein's inequality to obtain
\begin{eqnarray*}
\|K_2\|_{L^2} &\le& C\, \|\Delta_j u^{(2)}\|_{L^\infty} \, \|\nabla \Delta_j \theta\|_{L^2} \\
&\le& C\, \|\Delta_j \nabla u^{(2)}\|_{L^\infty} \, \|\Delta_j \theta\|_{L^2}\\
&\le& C\, \|\omega^{(2)}\|_{B^{0,\gamma}_{\infty,1}} \| \Delta_j \theta\|_{L^2}.
\end{eqnarray*}
Again, for $j\ge j_0$ with $j_0= 2$, we have
\begin{eqnarray*}
\|K_4\|_{L^2} &\le& C\, \|\Delta_j u^{(2)}\|_{L^\infty} \, \|S_{j-1} \nabla \theta\|_{L^2} \\
&\le& C\, 2^j\, \|\Delta_j \nabla u^{(2)}\|_{L^\infty} \, \sum_{m\le j-2} 2^{2(m-j)} 2^{-m} \|\Delta_m \theta\|_{L^2} \\
&\le& C\, 2^j\, \|\omega^{(2)}\|_{B^{0,\gamma}_{\infty,1}}  \|\theta\|_{B^{-1}_{2,\infty}}.
\end{eqnarray*}
\begin{eqnarray*}
\|K_5\|_{L^2} &\le& C\, 2^j \sum_{k\ge j-1} \|\Delta_k u^{(2)}\|_{L^\infty} \, \|\Delta_k \theta\|_{L^2} \\
&\le& C\, 2^j \sum_{k\ge j-1} 2^{-k} \|\Delta_k \nabla u^{(2)}\|_{L^\infty} \|\Delta_k \theta\|_{L^2} \\
&\le& C\, 2^j  \|\omega^{(2)}\|_{B^{0,\gamma}_{\infty,1}}  \|\theta\|_{B^{-1}_{2,\infty}}.
\end{eqnarray*}
Inserting the estimates in (\ref{mmm}), we find
\begin{eqnarray*}
\frac{d}{dt} \|\Delta_j \theta\|_{L^2} \le C\, 2^j \,\|v\|_{L^2}\, \|\theta^{(1)}\|_{B^{0,\gamma}_{\infty,1}}+ C\, 2^j  \|\omega^{(2)}\|_{B^{0,\gamma}_{\infty,1}}  \|\theta\|_{B^{-1}_{2,\infty}}.
\end{eqnarray*}
Integrating in time leads to
\begin{eqnarray*}
2^{-j}\|\Delta_j \theta(t)\|_{L^2} &\le& 2^{-j}\|\Delta_j \theta(0)\|_{L^2} + \, C\, \int_0^t \|v(s)\|_{L^2}\, \|\theta^{(1)}(s)\|_{B^{0,\gamma}_{\infty,1}} \,ds\\
&&  + \,C\,\int_0^t \|\omega^{(2)}(s)\|_{B^{0,\gamma}_{\infty,1}}  \|\theta(s)\|_{B^{-1}_{2,\infty}}\,ds.
\end{eqnarray*}
Taking the supremum with respect to $j$ yields (\ref{ggi}).
\end{proof}

\vskip .1in
\begin{lemma} \label{ggo2}
Assume that $v$ satisfies
\begin{equation} \label{veq}
\partial_t v + u^{(2)}\cdot\nabla v + u\cdot \nabla v^{(1)} + \sum_{j=1}^2 \left(u_j^{(2)} \nabla v_j + u_j \nabla v_j^{(1)}\right) + \Lambda v = -\nabla p + \theta e_2,
\end{equation}
Then
\begin{eqnarray*}
\|v(t)\|_{B^0_{2,\infty}} &\le& \|v(0)\|_{B^0_{2,\infty}}
+\|\theta(s)\|_{B^{-1}_{2,\infty}} \\
&& + \,C\, \int_0^t \|v(s)\|_{L^2} \, \left(\|\omega^{(1)}(s)\|_{B^{0,\gamma}_{\infty,1}} +\|\omega^{(2)}(s)\|_{B^{0,\gamma}_{\infty,1}}\right)\,ds.
\end{eqnarray*}
\end{lemma}

\begin{proof}
Let $k\ge -1$. After applying $\Delta_k$ to (\ref{veq}), taking the inner product with $\Delta_k v$ and integrating by parts, we find
\begin{eqnarray} \label{mdm}
\frac12 \frac{d}{dt}\|\Delta_k v\|_{L^2}^2 + 2^k \|\Delta_k v\|_{L^2}^2 = L_1 + L_2 + L_3 +L_4 + L_5,
\end{eqnarray}
where
\begin{eqnarray*}
&& L_1 = -\int \Delta_k v \cdot \Delta_k (u^{(2)}\cdot\nabla v) , \qquad L_2 = -\int \Delta_k v \cdot \Delta_k (u\cdot \nabla v^{(1)}), \\
&& L_3 = -\sum_{j=1}^2\int \Delta_k v \cdot \Delta_k (u_j^{(2)} \nabla v_j), \qquad L_4 = -\sum_{j=1}^2\int \Delta_k v \cdot \Delta_k (u_j \nabla v_j^{(1)}),\\
&& L_5 = -\int \Delta_k v_2 \cdot \Delta_k \theta.
\end{eqnarray*}
To estimate $L_1$, we decompose $\Delta_k (u^{(2)}\cdot\nabla v)$ as in (\ref{ddee}) and bound the components in a similar fashion as in the proof of Lemma \ref{ggo1}. We obtain after applying H\"{o}lder's inequality
$$
|L_1| \le C\, \|\Delta_k v\|_{L^2} \, \|v\|_{L^2} \|\omega^{(2)}\|_{B^{0,\gamma}_{\infty,1}}
$$
To handle $L_2$, we decompose $\Delta_k (u\cdot \nabla v^{(1)})$ as in (\ref{dd00}) and obtain
$$
|L_2| \le C\, \|\Delta_k v\|_{L^2} \, \|v\|_{L^2} \, \|\omega^{(1)}\|_{B^{0,\gamma}_{\infty,1}}.
$$
For $L_3$, we integrate by parts and use the divergence-free condition to obtain
$$
L_3 = \sum_{j=1}^2\int \Delta_k v \cdot \Delta_k (v_j  \nabla u_j^{(2)}).
$$
Decomposing $\Delta_k (v_j  \nabla u_j^{(2)})$ as in (\ref{dd00}) and estimate the resulting components as in the proof of Lemma \ref{ggo1}, we obtain
$$
|L_3| \le C\, \|\Delta_k v\|_{L^2} \, \|v\|_{L^2} \, \|\omega^{(2)}\|_{B^{0,\gamma}_{\infty,1}}.
$$
Clearly $L_4$ admits the same bound as $L_2$. $L_5$ can be bounded by applying H\"{o}lder's inequality
$$
|L_5| \le \|\Delta_k v\|_{L^2} \, \|\Delta_k \theta\|_{L^2} \, \le 2^k \|\Delta_k v\|_{L^2} \, \|\theta\|_{B^{-1}_{2,\infty}}.
$$
Inserting the estimates above in (\ref{mdm}), we find
$$
\frac{d}{dt}\|\Delta_k v\|_{L^2} + 2^k \|\Delta_k v\|_{L^2}
\le C\, \|v\|_{L^2} \, \left(\|\omega^{(1)}\|_{B^{0,\gamma}_{\infty,1}}
+ \|\omega^{(2)}\|_{B^{0,\gamma}_{\infty,1}}\right)
+ 2^k  \|\theta\|_{B^{-1}_{2,\infty}}.
$$
Integrating in time yields
\begin{eqnarray*}
\|\Delta_k v(t)\|_{L^2} &\le&  e^{-2^k t} \|\Delta_k v(0)\|_{L^2} + \int_0^t e^{-2^k (t-s)} 2^k  \|\theta(s)\|_{B^{-1}_{2,\infty}}\,ds  \\
&& + \,C \int_0^t e^{-2^k (t-s)} \|v(s)\|_{L^2} \, \left(\|\omega^{(1)}(s)\|_{B^{0,\gamma}_{\infty,1}} +\|\omega^{(2)}(s)\|_{B^{0,\gamma}_{\infty,1}}\right)\,ds.
\end{eqnarray*}
Therefore,
\begin{eqnarray*}
\|v(t)\|_{B^0_{2,\infty}} &\le& \|v(0)\|_{B^0_{2,\infty}} +\|\theta(s)\|_{B^{-1}_{2,\infty}} \\
&& + \,C\, \int_0^t \|v(s)\|_{L^2} \, \left(\|\omega^{(1)}(s)\|_{B^{0,\gamma}_{\infty,1}} +\|\omega^{(2)}(s)\|_{B^{0,\gamma}_{\infty,1}}\right)\,ds.
\end{eqnarray*}
This completes the proof of Lemma \ref{ggo2}.
\end{proof}

\vskip .4in
\appendix

\section{Besov spaces and Osgood inequality}
\label{Besov}

This appendix provides the definitions of some of the function spaces and related facts used in the previous sections. In addition, the Osgood inequality used in Sectioon \ref{uniq} is also provided here for the convenience of readers. Materials presented in this appendix can be found in several books and many papers (see, e.g., \cite{BCD,BL,RS,Tri}).

\vskip .1in
We start with several notation. $\mathcal{S}$ denotes
the usual Schwarz class and ${\mathcal S}'$ its dual, the space of
tempered distributions. ${\mathcal S}_0$ denotes a subspace of ${\mathcal
S}$ defined by
$$
{\mathcal S}_0 = \left\{ \phi\in {\mathcal S}: \,\, \int_{\mathbb{R}^d}
\phi(x)\, x^\gamma \,dx =0, \,|\gamma| =0,1,2,\cdots \right\}
$$
and ${\mathcal S}_0'$ denotes its dual. ${\mathcal S}_0'$ can be identified
as
$$
{\mathcal S}_0' = {\mathcal S}' / {\mathcal S}_0^\perp = {\mathcal S}' /{\mathcal P}
$$
where ${\mathcal P}$ denotes the space of multinomials.

\vspace{.1in} To introduce the Littlewood-Paley decomposition, we
write for each $j\in \mathbb{Z}$
\begin{equation}\label{aj}
A_j =\left\{ \xi \in \mathbb{R}^d: \,\, 2^{j-1} \le |\xi| <
2^{j+1}\right\}.
\end{equation}
The Littlewood-Paley decomposition asserts the existence of a
sequence of functions $\{\Phi_j\}_{j\in {\Bbb Z}}\in {\mathcal S}$ such
that
$$
\mbox{supp} \widehat{\Phi}_j \subset A_j, \qquad
\widehat{\Phi}_j(\xi) = \widehat{\Phi}_0(2^{-j} \xi)
\quad\mbox{or}\quad \Phi_j (x) =2^{jd} \Phi_0(2^j x),
$$
and
$$
\sum_{j=-\infty}^\infty \widehat{\Phi}_j(\xi) = \left\{
\begin{array}{ll}
1&,\quad \mbox{if}\,\,\xi\in {\Bbb R}^d\setminus \{0\},\\
0&,\quad \mbox{if}\,\,\xi=0.
\end{array}
\right.
$$
Therefore, for a general function $\psi\in {\mathcal S}$, we have
$$
\sum_{j=-\infty}^\infty \widehat{\Phi}_j(\xi)
\widehat{\psi}(\xi)=\widehat{\psi}(\xi) \quad\mbox{for $\xi\in {\Bbb
R}^d\setminus \{0\}$}.
$$
In addition, if $\psi\in {\mathcal S}_0$, then
$$
\sum_{j=-\infty}^\infty \widehat{\Phi}_j(\xi)
\widehat{\psi}(\xi)=\widehat{\psi}(\xi) \quad\mbox{for any $\xi\in
{\Bbb R}^d $}.
$$
That is, for $\psi\in {\mathcal S}_0$,
$$
\sum_{j=-\infty}^\infty \Phi_j \ast \psi = \psi
$$
and hence
$$
\sum_{j=-\infty}^\infty \Phi_j \ast f = f, \qquad f\in {\mathcal S}_0'
$$
in the sense of weak-$\ast$ topology of ${\mathcal S}_0'$. For
notational convenience, we define
\begin{equation}\label{del1}
\Delta_j f = \Phi_j \ast f, \qquad j \in {\Bbb Z}.
\end{equation}

\begin{define}
For $s\in {\Bbb R}$ and $1\le p,q\le \infty$, the homogeneous Besov
space $\mathring{B}^s_{p,q}$ consists of $f\in {\mathcal S}_0' $
satisfying
$$
\|f\|_{\mathring{B}^s_{p,q}} \equiv \|2^{js} \|\Delta_j
f\|_{L^p}\|_{l^q} <\infty.
$$
\end{define}

\vspace{.1in}
We now choose $\Psi\in {\mathcal S}$ such that
$$
\widehat{\Psi} (\xi) = 1 - \sum_{j=0}^\infty \widehat{\Phi}_j (\xi),
\quad \xi \in {\Bbb R}^d.
$$
Then, for any $\psi\in {\mathcal S}$,
$$
\Psi \ast \psi + \sum_{j=0}^\infty \Phi_j \ast \psi =\psi
$$
and hence
\begin{equation}\label{sf}
\Psi \ast f + \sum_{j=0}^\infty \Phi_j \ast f =f
\end{equation}
in ${\mathcal S}'$ for any $f\in {\mathcal S}'$. To define the inhomogeneous Besov space, we set
\begin{equation} \label{del2}
\Delta'_j f = \left\{
\begin{array}{ll}
0,&\quad \mbox{if}\,\,j\le -2, \\
\Psi\ast f,&\quad \mbox{if}\,\,j=-1, \\
\Phi_j \ast f, &\quad \mbox{if} \,\,j=0,1,2,\cdots.
\end{array}
\right.
\end{equation}
\begin{define}
The inhomogeneous Besov space $B^s_{p,q}$ with $1\le p,q \le \infty$
and $s\in {\Bbb R}$ consists of functions $f\in {\mathcal S}'$
satisfying
$$
\|f\|_{B^s_{p,q}} \equiv \|2^{js} \|\Delta'_j f\|_{L^p} \|_{l^q}
<\infty.
$$
\end{define}

\vskip .1in
The Besov spaces $\mathring{B}^s_{p,q}$ and $B^s_{p,q}$ with  $s\in (0,1)$ and $1\le p,q\le \infty$ can be equivalently defined by the norms
$$
\|f\|_{\mathring{B}^s_{p,q}}  = \left(\int_{\mathbb{R}^d} \frac{(\|f(x+t)-f(x)\|_{L^p})^q}{|t|^{d+sq}} dt\right)^{1/q},
$$
$$
\|f\|_{B^s_{p,q}}  = \|f\|_{L^p} + \left(\int_{\mathbb{R}^d} \frac{(\|f(x+t)-f(x)\|_{L^p})^q}{|t|^{d+sq}} dt\right)^{1/q}.
$$
When $q=\infty$, the expressions are interpreted in the normal way. Sometimes it is also necessary to generalize the Besov spaces to include an algebraic part of the modes.
\begin{define}
For $s, \gamma\in {\Bbb R}$ and $1\le p,q \le \infty$, the generalized Besov spaces $\mathring{B}^{s,\gamma}_{p,q}$ and $B^{s,\gamma}_{p,q}$ are defined by
$$
\|f\|_{\mathring{B}^{s,\gamma}_{p,q}} \equiv \|2^{js} (1+|j|)^\gamma \|\Delta_j f\|_{L^p} \|_{l^q}
<\infty,
$$
$$
\|f\|_{B^{s,\gamma}_{p,q}} \equiv \|2^{js} (1+|j|)^\gamma \|\Delta'_j f\|_{L^p} \|_{l^q}
<\infty.
$$
\end{define}

\vskip .1in
We have also used the space-time spaces defined below.
\begin{define}
For $t>0$, $s, \gamma\in {\Bbb R}$ and $1\le p,q,r \le \infty$, the space-time spaces
$\widetilde{L}^r_t \mathring{B}^{s,\gamma}_{p,q}$ and $\widetilde{L}^r_t B^{s,\gamma}_{p,q}$ are defined though the norms
$$
\|f\|_{\widetilde{L}^r_t\mathring{B}^{s,\gamma}_{p,q}} \equiv \|2^{js} (1+|j|)^\gamma \|\Delta_j f\|_{L^r_t L^p} \|_{l^q},
$$
$$
\|f\|_{\widetilde{L}^r_t B^{s,\gamma}_{p,q}} \equiv \|2^{js} (1+|j|)^\gamma \|\Delta'_j f\|_{L^r_t L^p} \|_{l^q}.
$$
\end{define}
These spaces are related to the classical space-time spaces $L^r_t \mathring{B}^{s,\gamma}_{p,q}$ and $L^r_t B^{s,\gamma}_{p,q}$ via the Minkowski inequality.

\vskip .1in
Many frequently used function spaces are special cases of Besov spaces. The following proposition lists some useful equivalence and embedding relations.
\begin{prop}
For any $s\in \mathbb{R}$,
$$
\mathring{H}^s \sim \mathring{B}^s_{2,2}, \quad H^s \sim B^s_{2,2}.
$$
For any $s\in \mathbb{R}$ and $1<q<\infty$,
$$
\mathring{B}^{s}_{q,\min\{q,2\}} \hookrightarrow \mathring{W}_{q}^s \hookrightarrow \mathring{B}^{s}_{q,\max\{q,2\}}.
$$
In particular, $\mathring{B}^{0}_{q,\min\{q,2\}} \hookrightarrow L^q \hookrightarrow \mathring{B}^{0}_{q,\max\{q,2\}}$.
\end{prop}

\vskip .1in
For notational convenience, we write $\Delta_j$ for
$\Delta'_j$. There will be no confusion if we keep in mind that
$\Delta_j$'s associated the homogeneous Besov spaces is defined in
(\ref{del1}) while those associated with the inhomogeneous Besov
spaces are defined in (\ref{del2}). Besides the Fourier localization operators $\Delta_j$, the partial sum $S_j$ is also a useful notation. For an integer $j$,
$$
S_j \equiv \sum_{k=-1}^{j-1} \Delta_k,
$$
where $\Delta_k$ is given by (\ref{del2}). For any $f\in \mathcal{S}'$, the Fourier transform of $S_j f$ is supported on the ball of radius $2^j$.

\vskip .1in
Bernstein's inequalities is a useful tool on Fourier localized functions and these inequalities trade integrability for derivatives. The following proposition provides Bernstein type inequalities for fractional derivatives.
\begin{prop}\label{bern}
Let $\alpha\ge0$. Let $1\le p\le q\le \infty$.
\begin{enumerate}
\item[1)] If $f$ satisfies
$$
\mbox{supp}\, \widehat{f} \subset \{\xi\in \mathbb{R}^d: \,\, |\xi|
\le K 2^j \},
$$
for some integer $j$ and a constant $K>0$, then
$$
\|(-\Delta)^\alpha f\|_{L^q(\mathbb{R}^d)} \le C_1\, 2^{2\alpha j +
j d(\frac{1}{p}-\frac{1}{q})} \|f\|_{L^p(\mathbb{R}^d)}.
$$
\item[2)] If $f$ satisfies
\begin{equation*}\label{spp}
\mbox{supp}\, \widehat{f} \subset \{\xi\in \mathbb{R}^d: \,\, K_12^j
\le |\xi| \le K_2 2^j \}
\end{equation*}
for some integer $j$ and constants $0<K_1\le K_2$, then
$$
C_1\, 2^{2\alpha j} \|f\|_{L^q(\mathbb{R}^d)} \le \|(-\Delta)^\alpha
f\|_{L^q(\mathbb{R}^d)} \le C_2\, 2^{2\alpha j +
j d(\frac{1}{p}-\frac{1}{q})} \|f\|_{L^p(\mathbb{R}^d)},
$$
where $C_1$ and $C_2$ are constants depending on $\alpha,p$ and $q$
only.
\end{enumerate}
\end{prop}

\vskip .1in
Finally we recall the Osgood inequality.
\begin{prop}\label{gronwall-2}
Let $\alpha(t)>0$ be a locally integrable function. Assume
$\omega(t)\ge 0$ satisfies
$$
\int_0^\infty \frac{1}{\omega(r)} dr = \infty.
$$
Suppose that $\rho(t)>0$ satisfies
$$
\rho(t)\le a +\int_{t_0}^t \alpha(s)\omega(\rho(s)) ds
$$
for some constant $a\ge 0$. Then if $a=0$, then $\rho\equiv 0$; if
$a>0$, then
$$
-\Omega(\rho(t))+\Omega(a)\le \int_{t_0}^t \alpha(\tau) d\tau,
$$
where
$$
\Omega(x)=\int_x^1 \frac{dr}{\omega(r)}.
$$
\end{prop}

\vskip .4in
\section*{Acknowledgements}
Chae's research was partially supported by NRF grant No.2006-0093854.
Wu's research was partially supported by NSF grant DMS 0907913 and the AT\&T
Foundation at Oklahoma State University.

\end{document}